\RequirePackage{amsthm}%

\documentclass[sn-mathphys,Numbered]{sn-jnl}

\usepackage{anyfontsize}
\usepackage{lmodern}%
\usepackage{graphicx}%
\usepackage{multirow}%
\usepackage{amsmath,amssymb,amsfonts,mathtools,bm}%
\usepackage{mathrsfs}%
\usepackage{enumerate}%
\usepackage[title]{appendix}%
\usepackage{xcolor}%
\usepackage{textcomp}%
\usepackage{manyfoot}%
\usepackage{booktabs}%
\usepackage{algorithm}%
\usepackage{algorithmicx}%
\usepackage{algpseudocode}%
\usepackage{listings}%
\usepackage{tikz}
\usepackage{svg}
\svgpath{{Images/}}
\graphicspath{{Images/}}

\usetikzlibrary{decorations.pathmorphing}\usetikzlibrary{patterns,arrows.meta,decorations.pathreplacing}

\newcommand{\N}{{\mathbb{N}}}
\newcommand{\R}{{\mathbb{R}}}

\newcommand{\Dc}{\mathcal{D}}
\newcommand{\Lc}{\mathcal{L}}

\newcommand{\Rc}{\mathcal{R}}
\newcommand{\Pc}{\mathcal{P}}
\newcommand{\Vc}{\mathcal{V}}
\newcommand{\Wc}{\mathcal{W}}

\DeclareMathOperator{\im}{im}

\DeclareMathOperator{\ext}{ext}

\DeclareMathOperator{\rank}{rank}

\newcommand{\bq}{\bm{\zeta}}
\newcommand{\bp}{\bm{\Gamma}}
\newcommand{\bv}{\bm{\varpi}}
\newcommand{\bF}{\bm{\tau}}

\newcommand{\q}{{r}}

\newcommand{\pos}{\mathrm{pos}}
\newcommand{\pot}{\mathrm{pot}}
\newcommand{\kin}{\mathrm{kin}}
\newcommand{\ddt}{{\textstyle\frac{\rm d}{{\rm d}t}}}

\newcommand{\setdef}[2]{\left\{ \, #1 \,\left\vert\vphantom{#1} \, #2 \, \right.\right\}}

\theoremstyle{thmstyleone}%
\newtheorem{theorem}{Theorem}
\newtheorem{proposition}[theorem]{Proposition}%
\newtheorem{lemma}[theorem]{Lemma}%

\theoremstyle{thmstyletwo}%
\newtheorem{remark}{Remark}%

\theoremstyle{thmstylethree}%
\newtheorem{definition}{Definition}%

\begin{document}

\title[Rigid multibody systems]{Port-Hamiltonian modeling of rigid multibody systems}


\author[1]{\fnm{Thomas} \sur{Berger}}\email{thomas.berger@math.upb.de}

\author[2]{\fnm{Ren\'e} \sur{Hochdahl}}\email{rene-christopher.hochdahl@tuhh.de}

\author*[3]{\fnm{Timo} \sur{Reis}}\email{timo.reis@tu-ilmenau.de}

\author[2]{\fnm{Robert} \sur{Seifried}}\email{robert.seifried@tuhh.de}


\affil[1]{\orgdiv{Institut f\"ur Mathematik}, \orgname{Universit\"at Paderborn}, \orgaddress{\street{Warburger Str. 100}, \city{Paderborn}, \postcode{33098}, \country{Germany}}}

\affil[2]{\orgdiv{Institut f\"ur Mechanik und Meerestechnik}, \orgname{Technische Universit\"at Hamburg}, \orgaddress{\street{Ei\ss endorfer Stra\ss e 42}, \city{Hamburg}, \postcode{21073}, \country{Germany}}}

\affil*[3]{\orgdiv{Institut f\"ur Mathematik}, \orgname{Technische Universit\"at Ilmenau}, \orgaddress{\street{Weimarer Stra\ss e 25}, \city{Ilmenau}, \postcode{98693}, \country{Germany}}}


\abstract{We employ a port-Hamiltonian approach to model nonlinear rigid multibody systems subject to both position and velocity constraints.  
Our formulation accommodates Cartesian and redundant coordinates, respectively, and captures kinematic as well as gyroscopic effects.  
The resulting equations take the form of nonlinear differential-algebraic equations that inherently preserve an energy balance.  
We show that the proposed class is closed under interconnection, and we provide several examples to illustrate the theory.
}

%
%
%
%
%




\maketitle


\section{Introduction}

Port-Hamiltonian system models encompass a broad class of nonlinear physical systems \cite{JvdS14, vdS17} and originate from port-based network modeling of complex dynamical systems in various physical domains, such as mechanical and electrical systems \cite{JvdS14, vdS13}. Modeling with port-Hamiltonian systems has garnered significant attention, and considerable progress has recently been made in port-Hamiltonian modeling of constrained dynamical systems, leading to differential-algebraic equations (DAEs) \cite{vdS10, BMXZ18, MMW18, MvdS18, MvdS20, vdS13, VvdS10a, GeHaRe20}.
In particular, the extension of the port-Hamiltonian framework to implicit energy storage \cite{MvdS18, MvdS20, MS22a, MS23} enables the modeling of a significantly broader range of physical systems with constraints, such as electrical circuits \cite{GeHaReVdS20}.


An inherent feature of port-Hamiltonian systems, in addition to their ability to maintain an energy balance, is the incorporation of a modular principle. This principle is built upon an elegant theory of power-preserving coupling, resulting once more in a port-Hamiltonian system \cite{JvdS14, CvdSB07, BKvdSZ10, SkJaEh23}. Historically, the theory of port-Hamiltonian systems has been developed based on the observation that principles from rational mechanics (see e.g.~\cite{Arn89}) apply to various physical domains. Over time, however, the theory has advanced well beyond its mechanical origins and has been further developed as an independent framework. In this article, we take a step back and analyze rigid multibody systems through the lens of port-Hamiltonian theory. 

Multibody systems~\cite{SchiehlenEberhard14,Woernle24,Shabana20} constitute a modern computer-oriented modeling approach for mechanical systems that undergo large rigid body translational and rotational motion. This approach is widely used in fields such as machine dynamics, vehicle dynamics, robotics or biomechanics. Multibody systems consist of rigid bodies which are connected by joints and bearings as well as coupling elements such as springs and dampers. Due to occurring finite rotational motions, the systems are inherently nonlinear. As a special case, multibody systems also include vibration systems of rigid bodies, which typically undergo only small linear motions. 

The Hamiltonian approach is typically formulated in terms of state variables consisting of position and momentum of each body. However, in the literature on mechanical systems, it is more common to use velocity~-- a so-called co-energy variable~-- instead of momentum. To bridge this gap, we employ the theory presented in~\cite{MvdS20}, which addresses implicit energy storage. This approach also facilitates the incorporation of both position and velocity constraints.

We begin by formulating translational motion in the port-Hamiltonian framework, considering systems of point masses with positions and velocities expressed in Cartesian coordinates. While this approach does not require much conceptual effort, it is limited from a practical perspective: when spatially expanded (in particular rotating) rigid bodies  are involved, their orientation must also be incorporated into the model.
In this case, the nonlinear relationship between angular velocity and the rotational parameters that determine the body's orientation must be taken into account. This introduces an additional layer of complexity compared to purely translational motion, 
which is also explored in this article within the port-Hamiltonian framework.

This article is structured as follows: After establishing basic notations, we provide a concise introduction to port-Hamiltonian systems in Section~\ref{sec:pHsys}. Subsequently, we consider multibody systems within this framework: point masses in Cartesian coordinates are discussed in Section~\ref{sec:rigidmks}. Multibody systems in redundant coordinates are considered in Section~\ref{sec:rigidmks2}. The interconnection of two systems via their ports is introduced in Section~\ref{sec:interconn} and it is shown that the class of port-Hamiltonian systems is closed under interconnection. A selection of illustrative examples are discussed in Section~\ref{Sec:Examples} and Conclusions are given in Section~\ref{Sec:Concl}. Some auxiliary results on Dirac structures and Lagrangian submanifolds are collected in Appendix~\ref{Sec:AppA}.

\subsection{Notation}

Throughout this article, the space $\R^n$ will be identified with matrices of size $n\times 1$. In particular, the Euclidean inner product of $z_1,z_2\in\R^n$ is given by $z_1^\top z_2$.

The vector space framework is often insufficient for describing physical systems, particularly those involving ideal constraints. Instead, the state evolution is represented within a {\em manifold}, a topological space that, at least locally, resembles the structure of a real vector space. This concept is further discussed in \cite{R21} in conjunction with port-Hamiltonian systems. For our purposes, it is adequate to consider the slightly simpler notion of submanifolds of a finite dimensional space $\R^n$.

\begin{definition}[Submanifold of $\R^n$]
Let $\mathcal{M}\subset\R^n$ be a subset. Then
$\mathcal{M}$ is called {\em differentiable submanifold of $\R^n$} (in this article just {\em submanifold of $\R^n$} for sake of brevity), if for all $x\in\mathcal{M}$, there exists a neighborhood $U_x\subset \R^n$ of~$x$, a number $k\in\N$, and a continuously differentiable mapping $f_x:U_x\to\R^k$, such that $f_x'(x)$ is surjective, 
and
    \[\mathcal{M}\cap U_x=\setdef{y\in U_x}{f_x(y)=0}.\]
\end{definition}

For a~submanifold $\mathcal{M}\subset\R^n$ and $x\in\mathcal{M}$, the implicit function theorem implies the existence of a number $k\in\N$, a neighborhood $V_x\subset\R^k$ of zero, and a continuously differentiable mapping $g_x:V_x\to\R^n$ satisfying the conditions $g_x(0)=x$, $g_x(V_x)\subset\mathcal{M}$ and $g_x'(0)$ has a~trivial kernel. Such a mapping $g_x$ possessing these properties is termed a {\em chart of $\mathcal{M}$ at~$x$}. 

Next we introduce the concept of the tangent space. To this end, we utilize  continuously differentiable curves on $\mathcal{M}$, i.e., continuously differentiable mappings $x:(-1,1)\to \R^n$ with $x(t)\in\mathcal{M}$ for all $t\in(-1,1)$. Defining the concept of the tangent space can be quite abstract in a general differential geometric context. However, the situation becomes significantly simpler when dealing with submanifolds of $\R^n$.
Loosely speaking, the tangent space $T_{x_0}\mathcal{M}$ at $x_0\in\mathcal{M}$ consists of all vectors that can be obtained by taking derivatives of smooth curves within $\mathcal{M}$ that pass through~$x_0$. 

\begin{definition}[Tangent space of a submanifold]
Let  $\mathcal{M}\subset\R^n$ be a~submanifold of $\R^n$.
The {\em tangent space} of $\mathcal{M}$ at $x_0\in\mathcal{M}$ is defined by
\[
T_{x_0}\mathcal{M} = \setdef{v \in\Wc}{{\parbox{9cm}{there exists a continuously differen\-tiable curve $x:(-1,1)\to \mathcal{M}$ with $x(0)=x_0$ and $\dot{x}(0)=v$}}}.
\]
The elements of $T_{x_0}\mathcal{M}$ are called {\em tangent vectors of $\mathcal{M}$ at $x_0$}.
\end{definition}
It can be observed from the definition of the tangent space that, for a chart $g_x: V_x \to \R^n$ of $\mathcal{M}$ at~$x$, we obtain $\im g_x'(0) = T_{x}\mathcal{M}$. In particular, $T_{x}\mathcal{M}$ is a~subspace.


\section{Port-Hamiltonian systems}\label{sec:pHsys}

We recall some fundamental concepts in port-Hamiltonian systems from \cite{MvdS18, MvdS20}. An essential concept to grasp is the notion of a \emph{Dirac structure}, which describes the power-preserving energy routing within the system. 


\begin{definition}[Dirac structure]\label{def-Dir}
A subspace $\mathcal D \subset \R^n\times \R^n$ is called a \emph{Dirac structure}, if
for all $f,e\in\R^n$, we have
\begin{equation*}
(f,e)\in \mathcal D\;\Longleftrightarrow\; \forall \, (\hat{f},\hat{e})\in \mathcal D:\; \hat{f}^\top e+f^\top\hat{e}=0.
\end{equation*}
\end{definition}
For $(f,e)\in\Dc$, we call $e$ an {\em effort} and $f$ is termed a \emph{flow}. By equipping $\R^n \times \R^n$ with the indefinite inner product 
\[
\begin{aligned}
\langle\!\langle\cdot,\cdot\rangle\!\rangle\;:&&\big(\R^n\times\R^n\big)\times\big(\R^n\times\R^n\big)&\to\R,\\
&&\big((f_1,e_1),(f_2,e_2)\big)&\mapsto \langle {f}_1,e_2\rangle+\langle f_2,{e}_1\rangle,
\end{aligned}\]
we can conclude that $\Dc\subset\R^n\times\R^n$ is a Dirac structure if, and only if, $\Dc=\Dc^{\bot\!\!\!\bot}$, where the latter denotes the orthogonal complement of $\Dc$ with respect to $\langle\!\langle\cdot,\cdot\rangle\!\rangle$.

Any Dirac structure $\Dc\subset\R^n\times\R^n$ is $n$-dimensional. Furthermore, for matrices $K,L\in\R^{n\times n}$, $\Dc=\im \begin{bmatrix} K & L\end{bmatrix}$ is a Dirac structure if, and only if, $\rank \begin{bmatrix} K & L\end{bmatrix}=n$ and $KL^\top+LK^\top=0$, see~\cite[Prop.~1.1.5]{Cou90}.

Dirac structures are oftentimes insufficent for being a~component in port-Hamiltonian modeling of rigid multibody systems. Therefore, we require the more comprehensive concept of a~{\em modulated Dirac structure} which is, loosely speaking, a~family of Dirac structures depending on a~parameter.

In a~more general setting, a modulated Dirac structure on a manifold
 $\mathcal M$ is defined as a {certain} subbundle of $T\mathcal M\oplus T'\mathcal M$ (representing the sum of the tangent bundle and cotangent bundle of $\mathcal{M}$) \cite[Def.~2.2.1]{Cou90}, as discussed in \cite{JvdS14}. This general manifold setup is not necessary for the class of multibody systems discussed in the present article. 

\begin{definition}[Modulated Dirac structure]\label{def-DirMod}
Let $U\subset\R^k$ be open. A~family $(\mathcal D_x)_{x\in U}$ of subspaces of  $\R^n\times \R^n$ is called a \emph{modulated Dirac structure}, if the following holds for all $x\in U$:
\begin{enumerate}[(a)]
\item\label{def-DirModa} $\mathcal D_x$ is a~Dirac structure.
\item\label{def-DirModb} There exists a~neighborhood $U_x\subset U$ of $x$ and a~family $(T_y)_{y\in U_x}$ of linear and bijective mappings $T_y:\R^n\to \mathcal D_y$, such that, for all $z\in\R^n$, the mapping
    \[y\mapsto T_yz\]
    is continuous from $U_x$ to $\R^n\times\R^n$.
    \end{enumerate}
\end{definition}
The conditions \eqref{def-DirModa} and \eqref{def-DirModb} outlined above imply that a modulated Dirac structure aligns with the definition of a vector bundle, as described in \cite[Chap.~III, \S~1]{Lang1999}, where~\eqref{def-DirModb} is referred to as {\em local trivialization}.



Next, we introduce a relation that describes the energy storage of the system, known as a {\em Lagrangian submanifold}. We note that the general definition of this concept, as found in \cite[p.\ 568]{Lee12}, is not required for the systems considered in this work. Instead of dealing with submanifolds of general manifolds, it suffices to consider submanifolds of $\R^n\times \R^n$. Typically, these manifolds are assumed to be smooth in the sense that the charts are infinitely often differentiable. However, for our purposes, we can relax this assumption and consider less smooth submanifolds.

\begin{definition}[Lagrangian submanifold]\label{def-Lag}
A submanifold $\mathcal L\subset \R^n\times \R^n$ is called \emph{Lagrangian submanifold} of $\R^n\times\R^n$, if for all $z\in \mathcal L$ and ${(v_1,v_2)}\in\R^n\times\R^n$ we have
\begin{equation}\label{lag-cond}
(v_1,v_2)\in T_z \mathcal L\;\Longleftrightarrow\; \forall\, (w_1,w_2)\in T_z \mathcal L:\ v_1^\top w_2- w_1^\top v_2=0.
\end{equation}
\end{definition}

In \cite[Prop.~22.12]{Lee12}, it is shown that for a smooth function $H:U\to\R^n$ defined on a simply connected domain $U\subset\R^n$, the set defined as
\[\mathcal L\coloneqq\setdef{(x,H(x))}{x\in U}\subset\R^n\times\R^n\]
is a Lagrangian submanifold if, and only if, $H$ is a gradient field. That is, $\nabla \mathcal{H}=H$ for some smooth function $\mathcal{H}:U\to\R$. In Appendix~\ref{Sec:AppA}, we present a generalization of this result, which will be used in the context of position constraints in mechanical systems. Additionally, Appendix~\ref{Sec:AppA} includes some results on special sets that form a~modulated Dirac structure. 

Another essential concept for port-Hamiltonian systems is that of a {\em(modulated) resistive relation}, which characterizes the internal energy dissipation within the system. This relation is defined on $\R^n\times \R^n$. The components of $(f_\Rc,e_\Rc)\in \Rc$ are called {\em resistive flows}~$f_\Rc$ and {\em resistive efforts}~$e_\Rc$, resp.\ \cite[Sec.~2.4]{JvdS14}.

\begin{definition}[(Modulated) resistive relation]\label{def-res}
A relation $\mathcal R\subset\R^n\times\R^n$ is called \emph{resistive}, if
\[\forall\, (f_\Rc,e_\Rc)\in\Rc:\  f_\Rc^\top e_\Rc\geq0.\]
Let $U\subset\R^k$ be open. A~family $(\mathcal R_x)_{x\in U}$ is called {\em modulated resistive relation}, if $\mathcal R_x\subset\R^n\times\R^n$ is a~resistive relation for all $x\in U$.
\end{definition}
Having established the definitions of (modulated) Dirac structures, Lagrangian submanifolds, and (modulated) resistive relations, we are now prepared to introduce port-Hamiltonian systems, cf.\ also Fig.~\ref{pH-sys}. Again, we note that this class can be defined in a more general setting involving manifolds \cite{MvdS20, JvdS14}. However, we simplify this to the specific setup required for our class of rigid multibody systems.

\begin{definition}[Port-Hamiltonian system]\label{def-pH}
A \emph{port-Hamiltonian (pH) system} is a~differential inclusion
\begin{equation*}
 \left(\begin{pmatrix}\dot x(t)\\f_\Rc(t)\\f_\Pc(t)\end{pmatrix},\begin{pmatrix}e_\Lc(t)\\e_\Rc(t)\\e_\Pc(t)\end{pmatrix}\right)\in\mathcal D_{x(t)},\quad
 \big(x(t),e_\Lc(t)\big)\in\mathcal L,\quad
  (f_\Rc(t),e_\Rc(t))\in\mathcal R_{x(t)},\label{eq:phinc}
\end{equation*}
where, for $U\subset\R^{n_{\Lc}}$ being open,
 \begin{itemize}
 \item  $\mathcal D=(\mathcal D_x)_{x\in U}$ with $\mathcal D_x\subset (\R^{n_{\Lc}}\times\R^{n_{\Rc}}\times\R^{n_{\Pc}})\times(\R^{n_{\Lc}}\times\R^{n_{\Rc}}\times\mathcal \R^{n_{\Pc}})$ is a~modulated Dirac structure (see Definition~\ref{def-Dir}),
 \item $\mathcal L\subset \mathcal \R^{n_{\Lc}}\times \R^{n_{\Lc}}$ is a~Lagrangian submanifold (see Definition~\ref{def-Lag}), and
  \item $\Rc=(\mathcal \Rc_x)_{x\in U}$ with $\mathcal \Rc_x\subset\mathcal \R^{n_{\Rc}}\times\mathcal \R^{n_{\Rc}}$ is a~modulated resistive relation (see Definition~\ref{def-res}).
\end{itemize}
The elements of $\R^{n_{\Lc}}$,  $\R^{n_{\Rc}}$, $\R^{n_{\Pc}}$ are, accordingly, called the \emph{energy-storing flows/efforts}, \emph{resistive flows/efforts} and  \emph{external flows/efforts}. 
\end{definition}
\begin{remark}
    It should be noted that our definition of port-Hamiltonian systems slightly differs from the one e.g.\ in \cite{JvdS14}, where the negative of $\dot{x}$ enters the Dirac structure, and the resistive relation fulfills $f_\Rc^\top e_\Rc\leq0$ for all $(f_\Rc,e_\Rc)\in \Rc_x$, $x\in U$. By substituting the flow with its negative, it is straightforward to establish a one-to-one correspondence between these two definitions. In practical terms, one approach views the system from the perspective of a consumer, while the other adopts a producer-centric viewpoint.
\end{remark}


\begin{figure}
	\centering
	\includegraphics[width=0.5\textwidth]{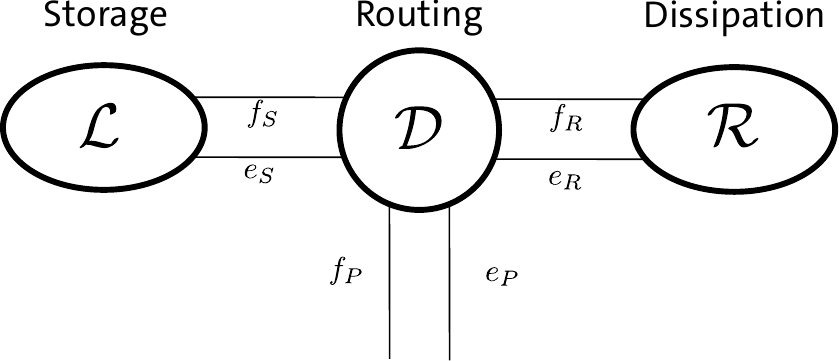}
	\caption{Visual representation of a port-Hamiltonian system.}
	 \label{pH-sys}
\end{figure}

\section{Newtownian mechanical systems: Point masses in Cartesian coordinates}\label{sec:rigidmks}

In this section, we consider mechanical systems consisting of point masses, with their positions expressed in conventional Cartesian coordinates, such as in $\R^2$ or $\R^3$. Loosely speaking, we consider multibody systems in which each rigid body has a negligible spatial expansion and possesses both potential and kinetic energy. The position and velocity of all masses are given by the vector-valued functions $\bm{\q}, \bm{v}: I \to \R^n$, where $I \subset \R$ denotes the considered time interval. The functions $\bm{\q}$ and $\bm{v}$ collectively encode the position coordinates and velocities, respectively, of all point masses in the system. Specifically, for a system with $p$ point masses in three-dimensional space, we have $n = 3p$, meaning each mass contributes three coordinates to the overall position vector.

We first present the modulated Dirac structure for this kind of systems. It contains velocity and force balances, as well as constraints formulated on velocity level. These velocity constraints are described by an equation of the form
\begin{equation*}
    A(\bm{\q}(t))\bm{v}(t)=0,\label{eq:nonhol}
\end{equation*}
where $A$ is a~continuous matrix-valued function with constant rank. Note that all or some of these constraints may be \emph{holonomic}~-- that is, they represent position constraints differentiated with respect to time~-- if they satisfy an integrability condition. Otherwise, they are referred to as \emph{nonholonomic}. 

Let $U_\pos \subset \R^n$ be an open set, representing the admissible positions of the point masses.
We consider the family that is modulated by the positions, namely $(\mathcal{D}_{\bm{\q}})_{{\bm{\q}}\in U_\pos}$ with
\begin{equation}
\mathcal{D}_{\bm{\q}}:=\setdef{ \begin{pmatrix}\bm{v}_{\Lc,f}\\\bm{F}_{\Lc,f}\\\bm{v}_{\Rc}\\\bm{v}_{\ext}\\\bm{F}_{\Lc,e}\\\bm{v}_{\Lc,e}\\\bm{F}_{\Rc}\\\bm{
F}_{\ext}\end{pmatrix}\in\R^{6n+2m}}{\,
{\parbox{7cm}{
$\bm{v}_{\Lc,f}=\bm{v}_{\Lc,e}=\bm{v}_{\Rc},\ A(\bm{\q}) \bm{v}_{\Lc,e}=0$,
\\[2mm] $\bm{v}_{\ext}=B(\bm{\q})^\top \bm{v}_{\Lc,e}$,\\[2mm] $\exists\,\bm{\mu}\in\R^\ell:$\\
$\bm{F}_{\Lc,f}+\bm{F}_{\Lc,e}+\bm{F}_{\Rc}+B(\bm{\q})\bm{F}_{\ext}+A(\bm{\q})^\top \bm{\mu}=0$}}\!\!}.\label{eq:rigidDirac}\end{equation}
Before proving that this constitutes a modulated Dirac structure, let us first discuss the practical meaning of the involved vectors and matrices.
The identical vectors $\bm{v}_{\Lc,f}, \bm{v}_{\Lc,e}, \bm{v}_{\Rc,e} \in \R^n$ represent the velocities of the point masses. One might wonder why three copies of the velocity appear in the Dirac structure. The reason is that velocity is needed in conjunction with potential energy, kinetic energy, and damping.
Furthermore, $\bm{F}_{\Lc,f} \in \R^n$ represents the inertial force, $\bm{F}_{\Rc} \in \R^n$ corresponds to the damping force, and $\bm{F}_{\Lc,e} \in \R^n$ denotes a~restoring force. The vectors $\bm{F}_{\ext}, \bm{v}_{\ext} \in \R^m$ collect the forces and velocities at the external ports, respectively.
Here, the matrix $B: U_{\pos} \to \R^{n \times m}$ describes the direction of the external forces. Moreover, as mentioned above, $A: U_{\pos} \to \R^{\ell \times n}$ encodes the velocity constraints. The term $A(\bm{\q})^\top \bm{\mu}$ represents the forces that ensure the fulfillment of the velocity constraints.

The following result states that $(\mathcal{D}_{\bm{\q}})_{{\bm{\q}}\in U_\pos}$ is a modulated Dirac structure. We do not present a proof at this point, but instead refer to Proposition~\ref{prop:dirac2}, which establishes a more general statement that encompasses the one below; we emphasize that Proposition~\ref{prop:dirac2} does not rely on any results from this section.

\begin{proposition}\label{prop:dirac1}
Let $U_{\pos}\subset\R^n$ be open, let $A:U_{\pos}\to\R^{\ell\times n}$,
$B:U_{\pos}\to\R^{n\times m}$ be continuous, and assume that $A$ has constant rank on $U_\pos$. Then the family $(\mathcal{D}_{\bm{\q}})_{\bm{\q}\in U_{\pos}}$ with $\mathcal{D}_{\bm{\q}}$ as in~\eqref{eq:rigidDirac} is a~modulated Dirac structure.
\end{proposition}
The damping elements are modeled using a modulated resistive structure. In the most general case, damping is described by a dissipative relation between the damping force and velocity. A typical application scenario is the consideration of  kinetic friction (see \cite{AH94,LN04}).
A dependence of this relation on the position vector $\bm{\q}$ may arise in non-homogeneous setups. For example, this can occur when a car moves on a road with varying surface properties.

Let us assume a damping model that is slightly simpler than the relational setup described above: 
Specifically, for $U_{\rm d} \subset U_{\pos} \times \R^n$, we consider a function $\bm{F}_{\rm d}: U_{\rm d} \to \R^{n}$ with
\[\forall\,(\bm{\q},\bm{v})\in U_{\rm d}:\quad \bm{v}^\top \bm{F}_{\rm d}(\bm{\q},\bm{v})\geq0.\]
Then it follows immediately that $\Rc = (\mathcal{R}_{\bm{\q}})_{\bm{\q}\in U_{\pos}}$ with
\begin{equation}
\mathcal{R}_{\bm{\q}}:=\setdef{\begin{pmatrix}\bm{F}\\\bm{v}\end{pmatrix}\in\R^{2{n}}}{\bm{F}=\bm{F}_{\rm d}(\bm{\q},\bm{v})}\label{eq:MKSresrel}
\end{equation}
is a~modulated resistive relation.\\
The remaining key component is the Lagrangian submanifold, which is, in the context of mechanical systems, responsible for storage of kinetic and potential energy. In particular, it encodes the restoring and inertial forces. 
Let $M \in \R^{n \times n}$ be a symmetric and positive definite matrix (representing the mass matrix),  
$\mathcal{V}_{\pot} : U_{\pos} \to \R$ a twice continuously differentiable mapping (representing the potential energy),  
and $c : U_{\pos} \to \R^k$ a twice continuously differentiable function (representing the holonomic constraints on the position $\q$),  
which vanishes on a non-empty subset of $U_{\pos}$ and whose Jacobian satisfies $\rank c'(\bm{\q}) = k$ for all $\bm{\q} \in U_{\pos}$.  
We then define
\begin{equation}
    \mathcal{L}:=\setdef{\!\begin{pmatrix}\bm{\q}\\\bm{p}\\\bm{F}\\\bm{v}\end{pmatrix}\in U_{\pos}\times \R^{3n}}{\;\parbox{55mm}{$M\bm{v}=\bm{p},\,c(\bm{\q})=0,\,$\\[2mm] $\exists\,\bm{\lambda}\in\R^k:\
    \bm{F}=\nabla \mathcal{V}_{\pot}(\bm{\q})+c'(\bm{\q})^\top\bm{\lambda}\!\!$}}.\label{eq:LagrMech0}
\end{equation}
Here, $\bm{p}$ stands for the momenta and $\nabla\Vc_{\pot}(\bm{\q})$ is the restoring force caused by potential energy storage. 
The term $c'(\bm{\q})^\top \bm{\lambda}$ represents the forces that ensure the fulfillment of the position constraints. 

We can conclude from Proposition~\ref{prop-gradient} in Appendix~\ref{Sec:AppA} (by setting $\mathcal{H}(\bm{\q},\bm{p})=\tfrac12\bm{p}^\top M^{-1}\bm{p}+\mathcal{V}_{\pot}(\bm{\q})$ and $d(\bm{\q},\bm{p})=c(\bm{\q})$) that $\Lc$ as in \eqref{eq:LagrMech0} is a~Lagrangian submanifold. 
Indeed, it is even a~Lagrangian submanifold, if $M$ is only positive semi-definite. This follows by an application of 
Proposition~\ref{prop-gradient} with $\mathcal{H}(\bm{\q},\bm{p})=\tfrac12\bm{p}^\top M^{+}\bm{p}+\mathcal{V}_{\pot}(\bm{\q})$, where $M^+\in\R^{n\times n}$ is the Moore-Penrose-inverse of $M$, and
\[d(\bm{\q},\bm{p})=\begin{pmatrix}
    c(\bm{\q})\\W\bm{p}
\end{pmatrix},\]
where, for $r=\rank M$, $W\in\R^{(n-r)\times n}$ is a~matrix with $\im W^\top =\ker M$.

The overall port-Hamiltonian model according to Definition~\ref{def-pH} is, for $(\mathcal{D}_{\bm{\q}})_{{\bm{\q}}\in U_\pos}$ as in \eqref{eq:rigidDirac}, $\Lc$ as in \eqref{eq:LagrMech0}, and $(\mathcal{R}_{\bm{\q}})_{{\bm{\q}}\in U_\pos}$ as in \eqref{eq:MKSresrel}, given by
\begin{align*}
\begin{pmatrix}\dot{\bm{\q}}(t)\\\dot{\bm{p}}(t)\\\bm{v}_{\Rc}(t)\\\bm{v}_{\ext}(t)\\\bm{F}(t)\\\bm{v}(t)\\\bm{F}_{\Rc}(t)\\\bm{F}_{\ext}(t)\end{pmatrix}&\in\mathcal{D}_{\bm{\q}(t)},\quad \begin{pmatrix}\bm{F}_{\Rc}(t)\\\bm{v}_{\Rc}(t)\end{pmatrix}\in\mathcal{R}_{\bm{\q}(t)},\quad
\begin{pmatrix}\bm{\q}(t)\\\bm{p}(t)\\\bm{F}(t)\\\bm{v}(t)\end{pmatrix}\in\mathcal{L}.
\end{align*}
The first relation gives
\begin{equation}
\begin{aligned}
\dot{\bm{\q}}(t)&=\bm{v}(t),\\
\dot{\bm{p}}(t)&=-\bm{F}(t)-\bm{F}_{\Rc}(t)-B(\bm{\q}(t))\bm{F}_{\ext}(t)-A(\bm{\q}(t))^\top \bm{\mu}(t),\\
0&=A(\bm{\q}(t))\dot{\bm{\q}}(t),\\
\bm{v}(t)&=\bm{v}_{\Rc}(t),\\
\bm{v}_{\ext}(t)&=B(\bm{\q}(t))^\top {\bm{v}}(t).
\end{aligned}\label{eq:mksDiracdyn}
\end{equation}
Now incorporating the Lagrangian submanifold and the modulated resistive relation, we have
\begin{equation*}
\begin{aligned}
M\bm{v}(t)&=\bm{p}(t),\\
\bm{F}(t)&=\nabla \mathcal{V}_\pot(\bm{\q}(t))+c'(\bm{\q}(t))^\top\bm{\lambda}(t),\\
0&=c(\bm{\q}(t)),\\
\bm{F}_{\Rc}(t)&=\bm{F}_{\rm d}(\bm{\q}(t),{\bm{v}_{\Rc}(t)}),
\end{aligned}
\end{equation*}
and inserting this into \eqref{eq:mksDiracdyn}, we obtain the overall model
\begin{equation}
\begin{aligned}
\dot{\bm{\q}}(t)&=\bm{v}(t),\\
\ddt M\bm{v}(t)&=-\nabla \mathcal{V}_\pot(\bm{\q}(t))-\bm{F}_{\rm d}(\bm{\q}(t),{ \bm{v}}(t))\\
&\quad\, -c'(\bm{\q}(t))^\top\bm{\lambda}(t)-A(\bm{\q}(t))^\top \bm{\mu}(t)-B(\bm{\q}(t))\bm{F}_{\ext}(t),\\
0&=c(\bm{\q}(t)),\\
0&=A(\bm{\q}(t)){\bm{v}}(t),\\
\bm{v}_{\ext}(t)&=B(\bm{\q}(t))^\top {\bm{v}}(t).
\end{aligned}\label{eq:mksdyn}
\end{equation}
We note that the total energy
\[\mathcal{H}(\bm{\q},\bm{v})=\tfrac12 \bm{v}^\top M\bm{v}+\Vc_\pot(\bm{\q})\]
fulfills, along the solutions of \eqref{eq:mksdyn}, 
\[\ddt \mathcal{H}(\bm{\q}(t),\bm{v}(t))=
-\bm{v}_{\ext}(t)^\top \bm{F}_{\ext}(t)-{\bm{v}}(t)^\top\bm{F}_{\rm d}(\bm{\q}(t),{ \bm{v}}(t))\leq -\bm{v}_{\ext}(t)^\top \bm{F}_{\ext}(t).
\]
That is, $\bm{v}_{\ext}(t)^\top \bm{F}_{\ext}(t)$ can be regarded as the power extracted from the system, whereas ${\bm{v}}(t)^\top\bm{F}_{\rm d}(\bm{\q}(t),{ \bm{v}}(t))$ is the dissipated power. 

\section{Multibody systems: Redundant coordinates}\label{sec:rigidmks2}

In this section, we consider a much more general class of multibody systems. In these general multibody systems, rigid bodies are considered. In contrast to point masses, a complete description of a rigid body includes not only its position but also its orientation, with the center of gravity often chosen as the reference point for the position. Accordingly, the Cartesian coordinates (three per body in space) must be extended by rotational parameters (at least three per body in space).

Different parameterizations such as unit-quaternions or Euler angles are possible within the port-Hamiltonian framework and have different advantages \cite{FuTaMa15}. In the following, we prefer to use Euler angles, as they require only three parameters and their associated singularity in spatial representation coincides with physical limitations of the systems discussed in Section~\ref{Sec:Examples}. Thus, six spatial coordinates are necessary for each rigid body (three in planar motion). It should be noted that these coordinates may depend on each other if {position} constraints are present. Since there are more coordinates available than necessary to describe the kinematics uniquely, these coordinates are called \emph{redundant coordinates}. In contrast, generalized coordinates are the minimal number (corresponding to the degree of freedom) of coordinates to uniquely describe the kinematics~\cite{SchiehlenEberhard14}. Generalized coordinates are used e.g.\ in the Newton-Euler formalism or the Lagrangian equation of 2nd order, yielding the equations of motion as an ordinary differential equation. In contrast, the use of redundant coordinates yield the equation of motion as a differential algebraic equation. Both approaches have case-dependent advantages and disadvantages.     

In the following, redundant coordinates are used to describe rigid multibody systems as port-Hamiltonian systems. With this approach, the kinematic relations between angular velocity and rotational parameters and the appearance of gyroscopic forces must be considered. {The latter appear when the translational and angular velocities are defined within a body-fixed coordinate frame. This approach is used to maintain a constant mass matrix.} Although this framework is more general than the one in the previous section and encompasses far more multibody systems than the previous approach, the class introduced in the following does not claim to cover all multibody systems with rigid components.

Subsequently, we use the following notation:

\begin{center}  \ \\\begin{tabular}{ll}
        $\bq$: & global positions, redundant coordinates \\ 
        $\bp$: & global momenta, \\ 
        $\bv$: & global velocities, velocity coordinates,  \\ 
        $\bF$: & global forces.\\[1mm]
    \end{tabular}

~\\
\end{center}
Note that there is no universally agreed-upon standard notation for this in the literature. The global positions, i.e. the redundant coordinates of all bodies, collect all positions (as discussed in the previous section) and orientation parameters of all rigid bodies. Consequently, $\bp$ may encompass both linear and angular momenta of all bodies, $\bv$ may consist of both linear and angular velocities of all bodies, and $\bF$ may include both forces and torques. The collection of the quantities of all bodies is reflected by the term {\em global}.

Again, we start with a function $\mathcal{V}_\pot: U_\pos \to \R$ representing the potential energy, where $U_\pos\subset \R^{n_\pot}$ is open. It is reasonable to assume that, in general, the number $n_\pot$ of variables determining the potential energy may differ from the number $n_\kin$ of variables governing the kinetic energy. This is evident in spring-mass-damper systems~\cite{Smith2002}, where each spring serves as a potential energy storage element, while each mass acts as a kinetic energy storage element. The kinetic energy is given by the term $\tfrac{1}{2}\bv^\top M \bv$, again with $M\in \R^{n_\kin \times n_\kin}$ being symmetric. This matrix incorporates inertia effects of all bodies, which is why it is referred to as the {\em global mass matrix}. For instance, it includes masses but may also contain inertia tensors.
As in the case of systems in Cartesian coordinates, assuming positive semi-definiteness is reasonable from a practical perspective, although it does not mathematically contribute to the findings of this article. Note that positive semi-definiteness becomes relevant when studying the existence of global solutions, which is not addressed here.

For now, we set aside considerations of energy and focus on the structural relationships between the involved global velocities  and global forces. These relationships are governed by a Dirac structure, which generalizes the one in \eqref{eq:rigidDirac} in several ways:
First, a kinematic relation between the velocities associated with kinetic and potential energies is incorporated. This is represented by the position-dependent matrix $Z: U_{\pos} \to \R^{n_\pot \times n_\kin}$.
Additionally, gyroscopic forces are introduced, which are captured by the pointwise skew-symmetric matrix $G: \R^{n_\kin}\to\R^{n_\kin \times n_{\kin}}$, whose argument consists of the globale momenta. We introduce the family $(\mathcal{D}_{(\bq,\bp)})_{{(\bq,\bp)}\in U_\pos\times\R^{n_\kin}}$, modulated by the global positions and  momenta, as
\begin{equation}
\mathcal{D}_{(\bq,\bp)}:=\setdef{ \begin{pmatrix}\bv_{\Lc,f}\\\bF_{\Lc,f}\\\bv_{\Rc}\\\bv_{\ext}\\\bF_{\Lc,e}\\\bv_{\Lc,e}\\\bF_{\Rc}\\\bF_{\ext}\end{pmatrix}\in\R^{2n_\pot+4n_\kin+2m}}{\,
{\parbox{5cm}{
$\bv_{\Lc,f}=Z(\bq)\bv_{\Lc,e},\,\bv_{\Rc}=\bv_{\Lc,e},$\\[2mm] $A(\bq) \bv_{\Lc,e}=0$,
\\[2mm] $\bv_{\ext}=B(\bq)^\top \bv_{\Lc,e}$,\\[2mm] $\exists\,\bm{\mu}\in\R^\ell:$\\
$\bF_{\Lc,f}+Z(\bq)^\top\bF_{\Lc,e}+G(\bp)\bv_{\Lc,e}$\\$+\bF_{\Rc}+B(\bq)\bF_{\ext}+A(\bq)^\top \bm{\mu}=0$}}\!\!}.\label{eq:rigidDirac2}\end{equation}

Next we show that this family is a modulated Dirac structure.

\begin{proposition}\label{prop:dirac2}
Let $U_{\pos}\subset\R^{n_\pot}$ be open, let
$A:U_{\pos}\to\R^{\ell\times n_{\kin}}$, $B:U_{\pos}\to\R^{n_\kin\times m}$, $G: \R^{n_\kin}\to\R^{n_\kin \times n_{\kin}}$, $Z: U_{\pos} \to \R^{n_\pot \times n_\kin}$ be continuous, and assume that $A$ has constant rank on $U_\pos$, and $G(\bp)$ is skew-symmetric for all $\bp\in\R^{n_\kin}$. Then the family $(\mathcal{D}_{(\bq,\bp)})_{{(\bq,\bp)}\in U_\pos\times\R^{n_\kin}}$ with $\mathcal{D}_{(\bq,\bp)}$ as in~\eqref{eq:rigidDirac2} is a~modulated Dirac structure.
\end{proposition}
\begin{proof}
{\em Step~1:} We show the statement under the additional assumption that $\ell=0$ (i.e., there are no constraints).  Denote the vector space as in \eqref{eq:rigidDirac2}, with additionally $\ell=0$, by $\hat{\mathcal{D}}_{(\bq,\bp)}$.
Then an image representation of this space is given by
\begin{equation*}
\hat{\mathcal{D}}_{(\bq,\bp)}=\im\begin{bmatrix}
0&Z(\bq)&0&0\\
-Z(\bq)^\top&-G(\bp)&-I_{n_\kin}&-B(\bq)\\
0&I_{n_\kin}&0&0\\
0&B(\bq)^\top&0&0\\
I_{n_\pot}&0&0&0\\
0&I_{n_\kin}&0&0\\
0&0&I_{n_\kin}&0\\
0&0&0&I_m
\end{bmatrix} = \im \begin{bmatrix} J(\bq,\bp)\\ I_{\hat n}\end{bmatrix},
\end{equation*}
where $\hat n = n_\pot + 2 n_\kin + m$ and $J:U_\pos\times\R^{n_\kin}\to \R^{\hat n\times \hat n}$ is continuous and pointwise skew-symmetric. Fix $(\bq,\bp)\in U_\pos\times\R^{n_\kin}$. Let $(f_i,e_i)\in \hat{\mathcal{D}}_{(\bq,\bp)}$, $i=1,2$, then there exist $g_1,g_2\in\R^{\hat n}$ such that 
\[
    \begin{pmatrix} f_i\\ e_i\end{pmatrix} = \begin{pmatrix} J(\bq,\bp) g_i\\ g_i\end{pmatrix},\quad i=1,2.
\]
From this it follows that
\[
    f_2^\top e_1 + f_1^\top e_2 = \eta_2^\top J(\bq,\bp)^\top \eta_1 + \eta_1^\top J(\bq,\bp)^\top \eta_2 =0.
\]
On the other hand, if for any $f,e\in\R^{\hat n}$ we have that, for all $(J(\bq,\bp)g,g)\in \hat{\mathcal{D}}_{(\bq,\bp)}$,
\[
   0 =  f^\top g + g^\top J(\bq,\bp)^\top e = g^\top \big( f - J(\bq,\bp) e\big),
\]
then $f = J(\bq,\bp) e$ as $g\in\R^{\hat n}$ is arbitrary. This shows that $\hat{\mathcal{D}}_{(\bq,\bp)}$ is a Dirac structure. Property~(b) of Definition~\ref{def-DirMod} follows from continuity of $(\bq,\bp)\mapsto J(\bq,\bp)$ with the linear and bijective maps $T_{(\bq,\bp)}:= \left[\begin{smallmatrix}  J(\bq,\bp)\\ I_{\hat n} \end{smallmatrix}\right]$.
Therefore, $(\hat{\mathcal{D}}_{(\bq,\bp)})_{{(\bq,\bp)}\in U_\pos\times\R^{n_\kin}}$ is a~modulated Dirac structure.

{\em Step~2:} We prove the general statement for $\ell>0$. Since $A$ has constant rank on $U_\pos$, we find that
\[E:U_\pos\times\R^{n_\kin}\to\R^{\ell\times(n_\pot+2n_\kin+m)},\ (\bq,\bp)\mapsto [0_{\ell\times n_\pot},A(\bq),0_{\ell\times (n_\kin+m)}]\]
has constant rank on $U_\pos\times\R^{n_\kin}$. Further, the structure of 
$\hat{\mathcal{D}}_{(\bq,\bp)}$ yields
\begin{multline*}
    \forall\,(\bq,\bp)\in U_{\pos}\times \R^{n_\kin}:\\\dim \Big(\hat{\mathcal{D}}_{(\bq,\bp)}\cap\big(\R^{n_\pot+2n_\kin+m}\times\ker E(\bq,\bp)\big)\Big)=n_\pot+2n_\kin+m-\rank A(\bq),
\end{multline*}
which is constant. As a~consequence, the assumptions of Proposition~\ref{prop:Dirnonhol} are fulfilled, and we can conclude that $(\Dc_{(\bq,\bp)})_{(\bq,\bp)\in U_{\pos}\times\R^{n_\kin}}$ is a~modulated Dirac structure.\hfill
\end{proof}

\begin{remark}
    Recall that in Section~\ref{sec:rigidmks} we mentioned that Proposition~\ref{prop:dirac1} would be a consequence of Proposition~\ref{prop:dirac2}. Indeed, this is the case for $n_\kin=n_\pot$, $Z\equiv I_{n_\kin}$ and $G\equiv 0_{n_\kin\times n_\kin}$.
\end{remark}

As in the previous section, we consider resistive relations defined by a~function. That is, for $U_{\rm d} \subset U_{\pos} \times \R^n$, we consider a function $\bF_{\rm d}: U_{\rm d} \to \R^{n_\kin}$ with
\[\forall\,(\bq,\bv)\in U_{\rm d}:\quad \bv^\top \bF_{\rm d}(\bq,\bv)\geq0.\]
This defines a~modulated resistive relation $\Rc = (\mathcal{R}_{\bq})_{\bq\in U_{\pos}}$ via
\begin{equation*}
\mathcal{R}_{\bq}:=\setdef{\begin{pmatrix}\bF\\\bv\end{pmatrix}\in\R^{2{n_\kin}}}{\bF=\bF_{\rm d}(\bq,\bv)}.\label{eq:MKSresrel2}
\end{equation*}
Basically, we consider the same Lagrangian submanifold as in \eqref{eq:LagrMech0}. The only difference is that we might have different dimensions for the position and momentum. 
As before, let $c: U_{\pos} \to \R^k$ denote the position constraints in the form $c(\bq) = 0$, and let $M \in \R^{n_\kin \times n_\kin}$ be the mass matrix, which is assumed to be symmetric. We assume that~$c$ as well as $\mathcal{V}_\pot:U_\pos\to\R$  are twice continuously differentiable and that $c$ vanishes in a~non-empty subset of $U_\pos$, and $c'(\bq)$ has full row rank for all $\bq\in U_\pos$. 
 The same argumentation as in the previous section then yields that
\begin{equation*}
    \mathcal{L}:=\setdef{\!\!\begin{pmatrix}\bq\\\bp\\\bF\\\bv\end{pmatrix}\in\R^{2n_{\pot}+2n_{\kin}}\!}{\parbox{55mm}{$M\bv=\bp,\,c(\bq)=0,$\\[2mm]$\exists\,\bm{\lambda}\in\R^k:\ 
    \bF=\nabla \mathcal{V}_\pot(\bq)+c'(\bq)^\top\bm{\lambda}\!\!$}}\label{eq:LagrMech}
\end{equation*}
is a~Lagrangian submanifold.

Let us now derive the equations for the port-Hamiltonian system governed by the modulated Dirac structure, modulated resistive relation, and Lagrangian submanifold introduced thus far. This is given by 
\begin{align*}
\begin{pmatrix}\dot{\bq}(t)\\\dot{\bp}(t)\\\bv_{\Rc}(t)\\\bv_{\ext}(t)\\\bF(t)\\\bv(t)\\\bF_{\Rc}(t)\\\bF_{\ext}(t)\end{pmatrix}&\in\mathcal{D}_{(\bq(t),\bp(t))},\quad \begin{pmatrix}\bF_{\Rc}(t)\\\bv_{\Rc}(t)\end{pmatrix}\in\mathcal{R}_{\bq(t)},\quad
\begin{pmatrix}\bq(t)\\\bp(t)\\\bF(t)\\\bv(t)\end{pmatrix}\in\mathcal{L}.
\end{align*}
The first relation gives
\[\begin{aligned}
\dot{\bq}(t)&=Z({\bq}(t))\bv(t),\\[1mm]
\dot{\bp}(t)&=-Z({\bq}(t))^\top\bF(t)-\bF_{\Rc}(t)-G(\bp(t))\bv(t)\\
&\quad\,-B(\bq(t))\bF_{\ext}(t)-A(\bq(t))^\top \bm{\mu}(t),\\[1mm]
0&=A(\bq(t)){\bv}(t),\\[1mm]
\bv(t)&=\bv_{\Rc}(t),\\[1mm]
\bv_{\ext}(t)&=B(\bq(t))^\top {\bv}(t).
\end{aligned}\]
Then, by the relations from the Lagrangian submanifold and the modulated resistive relation, we have
\[\begin{aligned}
M\bv(t)&=\bp(t),\\
\bF(t)&=\nabla \mathcal{V}_\pot(\bq(t))+c'(\bq(t))^\top\bm{\lambda}(t),\\
0&=c(\bq(t)),\\
\bF_{\Rc}(t)&=\bF_{\rm d}(\bq(t),{\bv_{\Rc}(t)}).
\end{aligned}
\]
Overall, this leads to the differential-algebraic system
\begin{equation}
\begin{aligned}
\dot{\bq}(t)&=Z({\bq}(t))\bv(t),\\[1mm]
 M\dot{\bv}(t)&=-Z({\bq}(t))^\top\nabla \mathcal{V}_\pot(\bq(t))-Z({\bq}(t))^\top c'(\bq(t))^\top\bm{\lambda}(t)
-\bF_{\rm d}(\bq(t),{\bv(t)})\\
&\quad\,-G(M\bv(t))\bv(t)-A(\bq(t))^\top \bm{\mu}(t)-B(\bq(t))\bF_{\ext}(t),\\[1mm]
0&=c(\bq(t)),\\[1mm]
0&=A(\bq(t)){\bv}(t),\\[1mm]
\bv_{\ext}(t)&=B(\bq(t))^\top {\bv}(t).
\end{aligned}\label{eq:mks2}
\end{equation}
Let us consider the total energy of the system, which is the sum of the kinetic energy $\tfrac{1}{2} \bv^\top M\bv$ and the potential energy $\Vc_\pot(\bq)$. Along the solutions of \eqref{eq:mks2}, the total energy satisfies
    \begin{align*}
    &\phantom{=}\ddt\big(\tfrac12\bv(t)^\top M\bv(t)+\mathcal{V}_\pot(\bq(t))\big)\\[2mm]
&={\bv}(t)^\top \ddt \big(M\bv(t)\big)+\dot{\bq}(t)^\top\nabla\mathcal{V}_\pot(\bq(t))\\
&={\bv}(t)^\top \Big(\ddt \big(M\bv(t)\big)+ Z({\bq}(t))^\top\nabla\mathcal{V}_\pot(\bq(t))\Big)\\
&\stackrel{\eqref{eq:mks2}}{=}- {\bv}(t)^\top \Big(Z({\bq}(t))^\top c'(\bq(t))^\top\bm{\lambda}(t)
+\bF_{\rm d}(\bq(t),{\bv(t)})
+G(M\bv(t))\bv(t)\\
&\qquad\qquad+A(\bq(t))^\top \bm{\mu}(t)+B(\bq(t))\bF_{\ext}(t)\Big).
\end{align*}
The velocity constraint gives $\bv(t)^\top A(\bq(t))^\top=0$,
skew-symmetry of $G(M\bv(t))$ implies $\bv(t)^\top G(M\bv(t))\bv(t)=0$. Moreover, 
\[\bv(t)^\top Z(\bq(t))^\top c'(\bq(t))^\top=\dot{\bq}(t)^\top c'(\bq(t))^\top=\ddt c(\bq(t))=0.\]
By further invoking $\bv_{\ext}(t)=B(\bq(t))^\top {\bv}(t)$, we see that the above energy balance reduces to 
    \begin{align*}
\ddt\big(\tfrac12\bv(t)^\top M\bv(t)+\mathcal{V}_\pot(\bq(t))\big)
=- {\bv}(t)^\top 
\bF_{\rm d}(\bq(t),{\bv(t)})
- {\bv}_{\ext}(t)^\top\bF_{\ext}(t),
\end{align*}
again with the interpretation that
$\bv_{\ext}(t)^\top \bF_{\ext}(t)$ is the power extracted from the system, whereas ${\bv}(t)^\top\bF_{\rm d}(\bq(t),{ \bv}(t))$ represents the dissipated power.

\section{Interconnection}\label{sec:interconn}

One of the biggest advantages of port-Hamiltonian models is their modular structure. That is, the class of port-Hamiltonian systems is closed under a~certain type of interconnection \cite{CvdSB07}. This type of interconnection is based on a~splitting of the external ports of two systems into those to be linked and those that are `truly external', that is
\[f_{\Pc i}(t)=\begin{pmatrix}
    f_{\Pc {\rm c},i}(t)\\f_{\Pc {\ext},i}(t)
\end{pmatrix},\, e_{\Pc i}(t)=\begin{pmatrix}
    e_{\Pc {\rm c},i}(t)\\e_{\Pc {\ext},i}(t)
\end{pmatrix}\in \R^{m_{\rm c}}\times \R^{m_{\ext,i}},\quad i=1,2.\]
The coupling relations are
\begin{equation}
f_{\Pc {\rm c},1}(t)=f_{\Pc {\rm c},2}(t),\quad e_{\Pc {\rm c},1}(t)=-e_{\Pc {\rm c},2}(t),\label{eq:couple}
\end{equation}
and the resulting external ports of the interconnected system are given by
\[f_{\Pc}(t)=\begin{pmatrix}
    f_{\Pc \ext,1}(t)\\f_{\Pc {\ext},2}(t)
\end{pmatrix},\, e_{\Pc}(t)=\begin{pmatrix}
    e_{\Pc {\ext},1}(t)\\e_{\Pc {\ext},2}(t)
\end{pmatrix}\in \R^{m_{\ext,1}}\times \R^{m_{\ext,2}}.\]
For completeness, note that \cite{CvdSB07} uses the interconnection rules $f_{\Pc {\rm c},1} = -f_{\Pc {\rm c},2}$ and $e_{\Pc {\rm c},1} = e_{\Pc {\rm c},2}$, which are equivalent to ours. However, we prefer the coupling relations in \eqref{eq:couple}, as they are more natural in the context of the class of multibody systems studied in this article.

For constant (i.e., unmodulated) Dirac structures, the coupling \eqref{eq:couple} again results in a Dirac structure, as shown in \cite{CvdSB07}.
For modulated Dirac structures, it remains unknown whether this type of coupling still results in a modulated Dirac structure, as it is unclear whether the local trivialization property holds for the interconnected system. In \cite{JaYo11,Jacobs2010}, additional conditions on the involved modulated Dirac structures have been imposed to ensure that the interconnection preserves the modulated Dirac structure property.

Our situation is somewhat more specific, as we focus on a subclass of port-Hamiltonian systems~-- namely, multibody systems as described in Section~\ref{sec:rigidmks2}. We demonstrate that, under a~rather mild additional assumption on the matrices~$A$ and~$B$, this class remains closed under interconnection via~\eqref{eq:couple}.

More precisely, we consider two multibody systems of type \eqref{eq:mks2}, where all involved matrices, functions, and variables are indexed by $i = 1,2$, depending on the respective system. We further partition the matrices $B_i$ by
\[B_i(\bq_i)=[B_{{\rm c},i}(\bq_i),B_{{\ext},i}(\bq_i)],\quad \;B_{{\rm c},i}\in \R^{n_{\kin,i}\times m_{\rm c}},\;B_{{\ext},i}\in \R^{n_{\kin,i}\times m_{\ext,i}},\quad i=1,2,\]
and the external forces and velocities by
\begin{multline*}
\bF_{\ext,i}(t)=\begin{pmatrix}\bF_{{\rm c},i}(t)\\\bF_{\widetilde{\ext},i}(t)\end{pmatrix},\,\bv_{\ext,i}(t)=\begin{pmatrix}\bv_{{\rm c},i}(t)\\\bv_{\widetilde{\ext},i}(t)\end{pmatrix},\\
\bF_{{\rm c},i}(t),\bv_{{\rm c},i}(t)\in \R^{m_{\rm c}},\; \bF_{\widetilde{\ext},i}(t),\bv_{\widetilde{\ext},i}(t)\in \R^{m_{\ext,i}}.
\end{multline*}
We show that under the assumption that 
\begin{equation}\label{eq:cond-interconn}
    \rank \begin{bmatrix}A_1(\bq_1)&0\\0&A_2(\bq_2)\\
B_{{\rm c},1}(\bq_1)^\top&- B_{{\rm c},2}(\bq_2)^\top\end{bmatrix}\equiv \text{const}
\end{equation}
on $U_{\pos,1}\times U_{\pos,2}$, the coupling
\[
\bF_{{\rm c},1}(t) = -\bF_{{\rm c},2}(t), \quad \bv_{{\rm c},1}(t) = \bv_{{\rm c},2}(t),
\]
which is induced by condition \eqref{eq:couple}, leads to an interconnected system which is again port-Hamiltonian. Note that the coupling corresponds to a connection between the two systems. More precisely, the velocities {of the connected ports are equal}, while the force on one side of the link acts as the counterforce on the other side. Altogether, this gives rise to the system


\begin{align*}
\dot{\bq}_1(t)&=Z_1({\bq}_1(t))\bv_1(t),\\[1mm]
\dot{\bq}_2(t)&=Z_2({\bq}_2(t))\bv_2(t),\\[1mm]
 M_1\dot{\bv}_1(t)&=-Z_1({\bq_1}(t))^\top\nabla \mathcal{V}_{\pot,1}(\bq_1(t))\\&\quad\,-Z_1({\bq}_1(t))^\top c_1'(\bq_1(t))^\top\bm{\lambda}_1(t)
-\bF_{{\rm d},1}(\bq_1(t),{\bv_1(t)})
\\&\quad\,-G_1(M_1\bv_1(t))\bv_1(t)-A_1(\bq_1(t))^\top \bm{\mu}_1(t)\\&\quad\,-B_{{\rm c},1}(\bq_1(t))\bF_{{\rm c},1}(t)- B_{\ext,1}(\bq_1(t))\bF_{\widetilde{\ext},1}(t),\\[1mm]
 M_2\dot{\bv}_2(t)&=-Z_2({\bq_2}(t))^\top\nabla \mathcal{V}_{\pot,2}(\bq_2(t))\\&\quad\,-Z_2({\bq}_2(t))^\top c_2'(\bq_2(t))^\top\bm{\lambda}_2(t)
+\bF_{{\rm d},2}(\bq_2(t),{\bv_2(t)})
\\&\quad\,-G_2(M_1(\bq_2(t))\bv_2(t))\bv_2(t)-A_2(\bq_2(t))^\top \bm{\mu}_2(t)\\&\quad\,+B_{{\rm c},2}(\bq_2(t))\bF_{{\rm c},1}(t)- B_{\ext,2}(\bq_2(t))\bF_{\widetilde{\ext},2}(t),\\[1mm]
0&=c_1(\bq_1(t)),\\[1mm]
0&=c_2(\bq_2(t)),\\[1mm]
0&=A_1(\bq_1(t)){\bv}_1(t),\\[1mm]
0&=A_2(\bq_2(t)){\bv}_2(t),\\[1mm]
0&=B_{{\rm c},1}(\bq_1(t))^\top{\bv}_1(t)-B_{{\rm c},2}(\bq_2(t))^\top{\bv}_2(t),\\[1mm]
\bv_{\widetilde{\ext},1}(t)&=B_{\ext,1}(\bq_1(t))^\top {\bv}_1(t),\\[1mm]
\bv_{\widetilde{\ext},2}(t)&=B_{\ext,2}(\bq_2(t))^\top {\bv}_2(t).
\end{align*}
It may appear quite complex at first glance, yet the system has an inherent structure. Specifically, it is of the form of \eqref{eq:mks2} with
\begin{align*}
n_\pot&= n_{\pot,1}+n_{\pot,2},\quad n_\kin= n_{\kin,1}+n_{\kin,2}, \quad m=m_{\ext,1}+m_{\ext,2},\\
\bq(t)&=\begin{pmatrix}
{\bq}_1\\
{\bq}_2
\end{pmatrix}, \;
\bv=\begin{pmatrix}
{\bv}_1\\
{\bv}_2
\end{pmatrix}, \;
\bp=\begin{pmatrix}
{\bp}_1\\
{\bp}_2
\end{pmatrix}, \\
\bv_{\ext}&=\begin{pmatrix}
\bv_{\widetilde{\ext},1}\\
\bv_{\widetilde{\ext},2}
\end{pmatrix}, \;\bF_{\ext}(t)=\begin{pmatrix}
\bF_{\widetilde{\ext},1}\\
\bF_{\widetilde{\ext},2}
\end{pmatrix},\\\mathcal{V}_\pot(\bq)&=\mathcal{V}_{\pot,1}(\bq_1)+\mathcal{V}_{\pot,2}(\bq_2),\\
Z(\bq)&=\begin{bmatrix}
Z_1({\bq}_1)&0\\0&Z_2({\bq}_2)
\end{bmatrix},\;\;M=\begin{bmatrix}
M_1&0\\0&M_2
\end{bmatrix},\\
G(\bp)&=\begin{bmatrix}
G_1({\bp}_1)&0\\0&G_2({\bp}_2)
\end{bmatrix},\;
\bF_{\rm d}(\bq,\bv)=\begin{pmatrix}\bF_{{\rm d},1}(\bq_1,\bv_1)\\\bF_{{\rm d},2}(\bq_2,\bv_2)\end{pmatrix},\\c(\bq)&=\begin{pmatrix}c_1(\bq_1)\\c_2(\bq_2)\end{pmatrix},\;
A(\bq)=\begin{bmatrix}A_1(\bq_1)&0\\0&A_2(\bq_2)\\
B_{{\rm c},1}(\bq_1)^\top&- B_{{\rm c},2}(\bq_2)^\top\end{bmatrix},\\
B(\bq)&=\begin{bmatrix}B_{\ext,1}(\bq_1) & 0 \\ 0 & B_{\ext,2}(\bq_2) \end{bmatrix}.
\end{align*}
This means that the interconnection rules cause for an additional velocity constraint $B_{{\rm c},1}(\bq_1(t))^\top{\bv}_1(t)-B_{{\rm c},2}(\bq_2(t))^\top {\bv}_2(t)=0$ and the corresponding Lagrange multiplier ${\bF}_{c,1}(t)$ also appears in the force balance.

By condition~\eqref{eq:cond-interconn} the matrix $A$ as above has constant 
rank on $U_{\pos,1}\times U_{\pos,2}$, hence it follows from the results in Section~\ref{sec:rigidmks2} that the interconnected system is again port-Hamiltonian. 

Finally, we note that the interconnection of more than two multibody systems can be reduced to the interconnection of two systems by an inductive approach.

\section{Examples}\label{Sec:Examples}

Here, we present three exemplary mechanical systems that are (partially) subject to constraints, incorporate gyroscopic effects, and include kinematic relationships. In the final example, we also illustrate the interconnection of mechanical systems in port-Hamiltonian form. {For a comprehensive explanation of the derivation of equations of motion for multibody systems, see~\cite{Woernle24}.}

\subsection{A differential drive robot in the plane}

\begin{figure}[t]
    \centering
    \def\svgwidth{0.7\textwidth}
    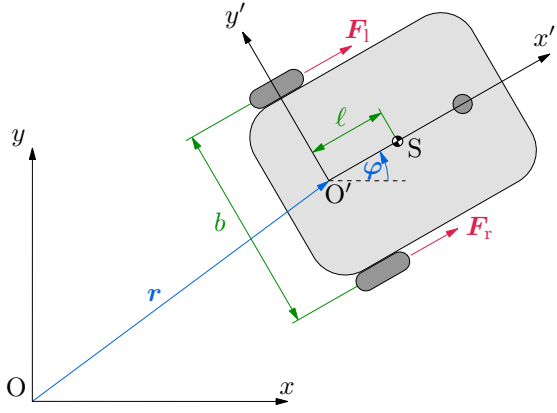
    \caption{Differential drive robot}
    \label{fig:DiffDriveRobot}
\end{figure}
Consider the differential drive robot depicted in Fig.~\ref{fig:DiffDriveRobot}.
Its body-fixed coordinate system is defined such that its $y'$-axis coincides with the axis of its wheels, while its $x'$-axis passes through its center of mass, whose $y'$-coordinate is $\ell\in\R$.
The robot is symmetric about its $x'$-axis and has a mass $m>0$ and a moment of inertia $\mathrm{I}_\mathrm{O}$ about the $z'$-axis at its origin. The identity
\begin{equation}
\mathrm{I}_\mathrm{O} = \mathrm{I}_\mathrm{S} + m\ell^2
\label{eq:momIS}\end{equation}
holds, where $\mathrm{I}_\mathrm{S}$ is the moment of inertia about the center of mass. We assume that $\mathrm{I}_\mathrm{S} > 0$, which means that the robot's mass is not entirely concentrated at the center of mass.
The system has three geometric degrees of freedom, described by the redundant coordinates
\[ \bq(t) = \begin{pmatrix} \bm{x}(t) \\ \bm{y}(t) \\ \bm{\varphi}(t) \end{pmatrix}, \] where $\bm{r}(t)=\Big(\begin{smallmatrix} \bm{x}(t) \\ \bm{y}(t) \end{smallmatrix}\Big)$ represents the position of the point $\rm O'$ in the inertial frame, and $\bm{\varphi}(t)$ denotes the orientation of the body-fixed frame relative to the inertial frame.
Moreover, the velocity vector of the robot at point $\rm O'$, expressed in the body-fixed frame, is defined as
\[ \bv(t) = \begin{pmatrix} \bm{v}_x(t) \\ \bm{v}_y(t) \\ \bm{\omega}(t) \end{pmatrix}. \]
The left and right wheels generate propulsion forces $\bm{F}_{\rm l}(t)$ and $\bm{F}_{\rm r}(t)$, respectively, acting in the $x'$-direction.
The wheels are spaced by a distance $b$.
As wheel slip is neglected, the robot cannot move along its $y'$-axis, leading to the velocity constraint $\bm{v}_y(t) = 0$. 
This constraint is enforced by the reaction force $\bm{\mu}(t) \in \R$.  
Furthermore, we introduce the wheel velocities as $\bm{v}_{\rm l}(t) = \bm{v}_x(t) - \tfrac{b}{2} \bm{\omega}(t)$ and  
$\bm{v}_{\rm r}(t) = \bm{v}_x(t) + \tfrac{b}{2} \bm{\omega}(t)$ of the left and right wheel, respectively.
The equations of motion in implicit form are then given by\begin{align*}
    \begin{pmatrix}
        \dot{\bm{x}}(t) \\ \dot{\bm{y}}(t) \\ \dot{\bm{\varphi}}(t)
    \end{pmatrix} &= \begin{bmatrix}
   \cos(\bm{\varphi}(t)) & -\sin(\bm{\varphi}(t)) & 0 \\        \sin(\bm{\varphi}(t)) & \cos(\bm{\varphi}(t)) & 0 \\
        0 & 0 & 1
    \end{bmatrix} \begin{pmatrix}
        \bm{v}_x(t) \\
        \bm{v}_y(t) \\
        \bm{\omega}(t)
    \end{pmatrix}, \\
    \begin{bmatrix}
        m & 0 & 0 \\
        0 & m & m\,\ell \\
        0 & m\,\ell & \mathrm{I}_{\rm O}
    \end{bmatrix} \begin{pmatrix}
        \dot{\bm{v}}_x(t) \\
        \dot{\bm{v}}_y(t) \\
        \dot{\bm{\omega}}(t)
    \end{pmatrix} &= {\begin{bmatrix}
    0 & 0 & m (\bm{v}_y+\ell\bm{\omega}) \\
    0 & 0 & -m \bm{v}_x \\
    -m (\bm{v}_y+\ell\bm{\omega}) & m \bm{v}_x & 0
\end{bmatrix}}
    \begin{pmatrix}
        \bm{v}_x(t) \\
        \bm{v}_y(t) \\
        \bm{\omega}(t)
    \end{pmatrix}\\&\quad\
    +\begin{bmatrix}
        0 \\ 1 \\ 0
    \end{bmatrix} \bm{\mu}(t) + \begin{bmatrix}
        1&1 \\
        0&0 \\
        -\frac{b}{2}&\frac{b}{2}
    \end{bmatrix} 
    \begin{pmatrix}
    \bm{F}_{\rm l}(t)\\\bm{F}_{\rm r}(t)\end{pmatrix},\\
    0 &= \begin{bmatrix}
        0 & 1 & 0
    \end{bmatrix} \begin{pmatrix}
        \bm{v}_x(t) \\
        \bm{v}_y(t) \\
        \bm{\omega}(t)
    \end{pmatrix},\\
    \begin{pmatrix}
        \bm{v}_{\rm l}(t)\\\bm{v}_{\rm r}(t)
    \end{pmatrix}
&=\begin{bmatrix}
        1&0&-\frac{b}{2} \\
        1&0&\phantom{-}\frac{b}{2}
    \end{bmatrix} 
\begin{pmatrix}
        \bm{v}_x(t) \\
        \bm{v}_y(t) \\
        \bm{\omega}(t)
    \end{pmatrix}.
\end{align*}
We neither have any position constraints nor damping, and the potential energy is constant. That is, we have a~system \eqref{eq:mks2} with $k=0$, $\mathcal{V}_\pot\equiv0$, $\bF_{\rm d}\equiv0$ and
\begin{align*}
&\bq=\begin{pmatrix}
        \bm{x} \\ \bm{y} \\ \bm{\varphi}
    \end{pmatrix},\;\;
\bm{Z}(\bq)=\begin{bmatrix}
   \cos(\bm{\varphi}) & -\sin(\bm{\varphi}) & 0 \\        \sin(\bm{\varphi}) & \cos(\bm{\varphi}) & 0 \\
        0 & 0 & 1
    \end{bmatrix},\;\; {M}=\begin{bmatrix}
        m & 0 & 0 \\
        0 & m & m\,\ell \\
        0 & m\,\ell & \mathrm{I}_{\rm O}
    \end{bmatrix},\;\; {B}(\bq)=\begin{bmatrix}
        1&1 \\
        0&0 \\
        -\frac{b}{2}&\frac{b}{2}
    \end{bmatrix}, \\ &{A}(\bq)=\begin{bmatrix}
        0 & 1 & 0
    \end{bmatrix}.\end{align*}
Using \eqref{eq:momIS} together with $m>0$ and $\mathrm{I}_\mathrm{S}>0$, we obtain that $\bm{M}$ is positive definite, and an inversion of $\bm{M}$ yields that the gyrator matrix can be written as
\[ \bm{G}(\bp)=m\frac{L-\ell \bm{p}_y}{\mathrm{I}_S}\begin{bmatrix}
     0&1&\ell\\
 -1&0&0\\-\ell&0&0
 \end{bmatrix},\quad\text{ where }\quad\bp=\begin{pmatrix}
    \bm{p}_x\\\bm{p}_y\\\bm{L}
 \end{pmatrix}.\]
Note that, in the above model, the velocity constraint can be resolved, which leads to the simplified system
\begin{align*}
    \begin{pmatrix}
        \dot{\bm{x}}(t) \\ \dot{\bm{y}}(t) \\ \dot{\bm{\varphi}}(t)
    \end{pmatrix} &= \begin{bmatrix}
        \cos(\bm{\varphi}(t)) & 0 \\
        \sin(\bm{\varphi}(t)) & 0 \\
        0 & 1
    \end{bmatrix} \begin{pmatrix}
        \bm{v}_x(t) \\ \bm{\omega}(t)
    \end{pmatrix}, \\
    \begin{bmatrix}
        m & 0 \\
        0 & \mathrm{I}_{\rm O}
    \end{bmatrix} \begin{pmatrix}
        \dot{\bm{v}}_x(t) \\ \dot{\bm{\omega}}(t)
    \end{pmatrix} &= \begin{bmatrix}
       0& m\bm{\omega}(t)\ell \\
        -m\bm{\omega}(t)\ell&0
    \end{bmatrix} \begin{pmatrix}
        \bm{v}_x(t) \\ \bm{\omega}(t)
    \end{pmatrix}+\begin{bmatrix}
        1&1 \\
        -\frac{b}{2}&\frac{b}{2}
    \end{bmatrix} 
    \begin{pmatrix}
    \bm{F}_{\rm l}(t)\\\bm{F}_{\rm r}(t)\end{pmatrix},\\
        \begin{pmatrix}
        \bm{v}_{\rm l}(t)\\\bm{v}_{\rm r}(t)
    \end{pmatrix}
&=\begin{bmatrix}
        1&-\frac{b}{2} \\
1&\phantom{-}\frac{b}{2}
    \end{bmatrix}
\begin{pmatrix}
        \bm{v}_x(t) \\
        \bm{\omega}(t)
    \end{pmatrix},
\end{align*}
which is again of the form \eqref{eq:mks2}, now without any explicitly stated velocity constraints.

\subsection{Gyroscope}

Consider the gimbal-mounted symmetrical gyroscope as depicted in Fig.~\ref{fig:MagnusKreisel}.
\begin{figure}[b]
    \centering
    \def\svgwidth{0.65\textwidth}
    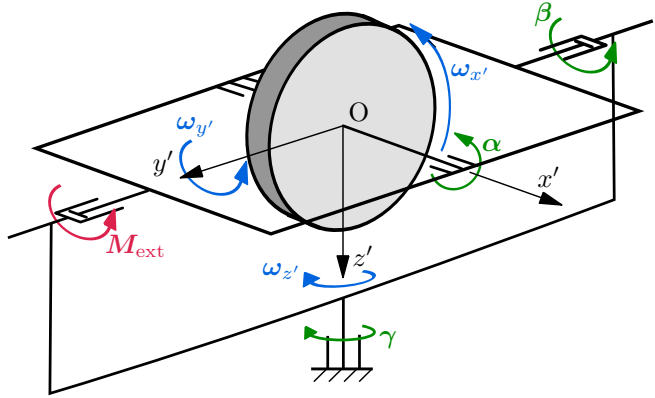
    \caption{Gyroscope}
    \label{fig:MagnusKreisel}
\end{figure}
The frames of the gimbal suspension are assumed to be massless.
The gyroscope's disk has a uniform mass density, with total mass $m > 0$, radius $r > 0$, and width $w > 0$.
The three axes of rotation intersect at the center of gravity $\mathrm{O}$ of the gyroscope.
Therefore, its inertia tensor is given by
\begin{equation}   \mathrm{I}_{\rm O} = \begin{bmatrix}
        \mathrm{I}_{{\rm O},x'} & & \\
        & \mathrm{I}_{{\rm O},y'} & \\
        & & \mathrm{I}_{{\rm O},z'} \\
    \end{bmatrix} = \frac{m}{12}\,\begin{bmatrix}
        6\,r^2 & & \\
        & 3\,r^2+w^2 & \\
        & & 3\,r^2+w^2 \\
    \end{bmatrix} .
\label{eq:intensor}
\end{equation}
We further assume that an external torque $\bm{M}_{\rm ext}(t)$ acts on the second axis {causing gyroscopic precession \cite{Arnold61}}. Referring to Fig.~\ref{fig:MagnusKreisel}, which depicts the Euler angles $\bm{\alpha}(t)$, $\bm{\beta}(t)$, and $\bm{\gamma}(t)$, we find that the angular velocities along the body-fixed axes $x'$, $y'$, and $z'$ are related to the time derivatives of the Euler angles by
\begin{subequations}
\begin{equation}
\begin{pmatrix}
    \dot{\bm{\alpha}}(t) \\
    \dot{\bm{\beta}}(t) \\
    \dot{\bm{\gamma}}(t)
\end{pmatrix}
=
\frac{1}{\cos(\bm{\beta}(t))}\!\!
\begin{bmatrix}
    \cos(\bm{\beta}(t)) & \sin(\bm{\alpha}(t))\,\sin(\bm{\beta}(t)) & \cos(\bm{\alpha}(t))\,\sin(\bm{\beta}(t)) \\
    0 & \cos(\bm{\alpha}(t))\,\cos(\bm{\beta}(t)) & -\sin(\bm{\alpha}(t))\,\cos(\bm{\beta}(t)) \\
    0 & \sin(\bm{\alpha}(t)) & \cos(\bm{\alpha}(t))
\end{bmatrix}\!\!
\begin{pmatrix}
    \bm{\omega}_{x'}(t) \\
    \bm{\omega}_{y'}(t) \\
    \bm{\omega}_{z'}(t)
\end{pmatrix}\!.\label{eq:gyrkin}
\end{equation}
Note that the singularity at $\bm{\beta} = \pm \frac\pi2$ describes the state in which gimbal lock appears~\cite{Arnold61}.
The kinetic equations of motion of the gyroscope read
\begin{multline}
     \begin{bmatrix}
        \mathrm{I}_{{\rm O},x'} & & \\
        & \mathrm{I}_{{\rm O},y'} & \\
        & & \mathrm{I}_{{\rm O},z'} \\
    \end{bmatrix} \begin{pmatrix}
        \dot{\bm{\omega}}_{x'}(t) \\ \dot{\bm{\omega}}_{y'}(t) \\ \dot{\bm{\omega}}_{z'}(t)
    \end{pmatrix} \\= \begin{bmatrix}
        0&-\mathrm{I}_{{\rm O},z'}\bm{\omega}_{z'}(t)&\phantom{-}\mathrm{I}_{{\rm O},y'}\bm{\omega}_{y'}(t)\\\phantom{-}
\mathrm{I}_{{\rm O},z'}\bm{\omega}_{z'}(t)&0&-\mathrm{I}_{{\rm O},x'}\bm{\omega}_{x'}(t)\\
-\mathrm{I}_{{\rm O},y'}\bm{\omega}_{y'}(t)&\phantom{-}\mathrm{I}_{{\rm O},x'}\bm{\omega}_{x'}(t)&0
    \end{bmatrix}
    \begin{pmatrix}
        {\bm{\omega}}_{x'}(t) \\ {\bm{\omega}}_{y'}(t) \\ {\bm{\omega}}_{z'}(t)
    \end{pmatrix}+  \begin{bmatrix}
        0 \\ \phantom{-}\cos(\bm{\alpha}(t)) \\ -\sin(\bm{\alpha}(t))
    \end{bmatrix}\bm{M}_{\rm ext}(t).
\end{multline}
By inverting the matrix in \eqref{eq:gyrkin}, we obtain that the angular velocity at the second axis is given by
\begin{equation}
\bm{\omega}_{\ext}(t)=
    \begin{bmatrix}
        0 &\cos(\bm{\alpha}(t)) & -\sin(\bm{\alpha}(t))
    \end{bmatrix}
    \begin{pmatrix}
        {\bm{\omega}}_{x'}(t) \\ {\bm{\omega}}_{y'}(t) \\ {\bm{\omega}}_{z'}(t)
\end{pmatrix}.\end{equation}
\label{eq:Gyr_alt}
\end{subequations}
Due to the mounting, the position of the center of gravity is fixed in space and the linear momentum equation vanishes. Taking this into account, the gyroscope can be written as a~system of the form \eqref{eq:mks2} with no constraints or damping, and with trivial potential energy (i.e., $\bF_{\rm d}\equiv 0$ and $\mathcal{V}_\pot\equiv0$). The redundant coordinates, velocities and external force are given by
\[\bq(t)=\begin{pmatrix}
    {\bm{\alpha}}(t) \\
    {\bm{\beta}}(t) \\
    {\bm{\gamma}}(t)
\end{pmatrix},\quad\bv(t)=\begin{pmatrix}
    {\bm{\omega}}_{x'}(t) \\
    {\bm{\omega}}_{y'}(t) \\
    {\bm{\omega}}_{z'}(t)
\end{pmatrix},\quad
\bF_{\ext}(t)=\bm{M}_{\ext}(t),\quad
\bv_{\ext}(t)=\bm{\omega}_{\ext}(t).\]
Further, the mass matrix is ${M}=\mathrm{I}_{\rm O}$ with $\mathrm{I}_{\rm O}\in\R^{3\times 3}$ as in \eqref{eq:intensor}, and, for $\bp=(\bm{L}_{x'},\bm{L}_{y'},\bm{L}_{z'})$,
the kinematic, gyrator and external force matrix, resp., read
\begin{align*}
    {Z}(\bq)&=\frac{1}{\cos(\bm{\beta})}
\begin{bmatrix}
    \cos(\bm{\beta}) & \sin(\bm{\alpha})\,\sin(\bm{\beta}) & \cos(\bm{\alpha})\,\sin(\bm{\beta}) \\
    0 & \cos(\bm{\alpha})\,\cos(\bm{\beta}) & -\sin(\bm{\alpha})\,\cos(\bm{\beta}) \\
    0 & \sin(\bm{\alpha}) & \cos(\bm{\alpha})
\end{bmatrix},\\
{G}(\bp)&=\begin{bmatrix}
        0&-\bm{L}_{z'}&\phantom{-}\bm{L}_{y'}\\\phantom{-}
\bm{L}_{z'}&0&-\bm{L}_{x'}\\
-\bm{L}_{y'}&\phantom{-}\bm{L}_{x'}&0
    \end{bmatrix},\quad 
{B}(\bq)=\begin{bmatrix}
        0 \\ \phantom{-}\cos(\bm{\alpha}) \\ -\sin(\bm{\alpha})
    \end{bmatrix}.
\end{align*}

\subsection{Slider Crank}
To illustrate port-Hamiltonian interconnection, we consider a planar slider-crank mechanism as shown in Fig.~\ref{fig:SliderCrank}. \begin{figure}[b]
    \centering
    \def\svgwidth{0.6\textwidth}
    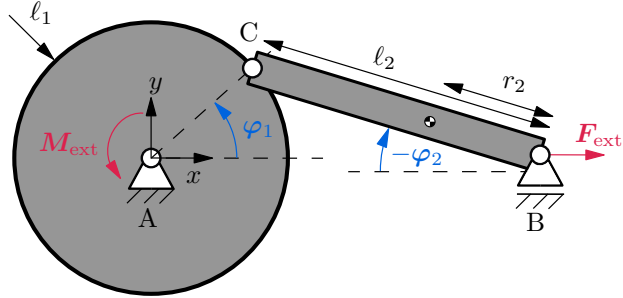
    \caption{Slider Crank}
    \label{fig:SliderCrank}
\end{figure}

The configuration consists of a crank (on the left), a slider and a connecting rod (on the right). For simplicity, the slider and the rod are modeled as one rigid body. 
We neither incorporate damping nor gravitational forces. The components are first modeled separately and then interconnected in the port-Hamiltonian framework. 


We describe the dynamics of the crank about the fixed point $\bm{A}$, around which it pivots. Consequently, there is no need to consider any translational coordinates. The crank's radius is $\ell_1$, and its orientation is described by the angle $\bm{\varphi}_1(t) \in \R$, and its angular velocity is denoted by $\bm{\omega}_1(t) \in \R$. The moment of inertia of the crank about point $\bm{A}$ is given by $\mathrm{I}_{1,\mathrm{A}}$. An external torque $\bm{M}_{\mathrm{ext}}(t) \in \R$ is applied at point $\bm{A}$ and forms an external port together with $\bm{\omega}_{\mathrm{ext}}(t) := \bm{\omega}_1(t)$. At the opposite end, the velocity $\bm{v}_{1,\mathrm{C}}(t) \in \R^2$ and the force $\bm{F}_{1,\mathrm{C}}(t) \in \R^2$ constitute another port, which is used to interconnect with the rod.
The dynamics are described by
\begin{align*}
\dot{\bm{\varphi}}_1(t)&=\bm{\omega}_1(t),\\
    \mathrm{I}_{1,\rm A}\,\dot{\bm{\omega}}_1(t) &= 
\bm{M}_{\ext}(t)+
\begin{bmatrix}
-\ell_1\,\sin(\bm{\varphi}_1(t)), & \ell_1\,\cos(\bm{\varphi}_1(t))
\end{bmatrix} \bm{F}_{1,\rm C}(t),\\
\bm{\omega}_{\ext}(t)&=\bm{\omega}_1(t),\\
\bm{v}_{1,\rm C}(t)&=\begin{bmatrix}
-\ell_1\,\sin(\bm{\varphi}_1(t)) \\ \phantom{-}\ell_1\,\cos(\bm{\varphi}_1(t))
\end{bmatrix} \bm{\omega}_1(t).
\end{align*}
The system is of the form \eqref{eq:mks2}, with mass matrix ${M} = \mathrm{I}_{1,\rm C}$ and kinematics matrix $\bm{Z}=1$. The motion does not have any constraints. Moreover, the potential energy and the damping are trivial, and $n_\kin=n_\pot=1$. Furthermore, the gyrator matrix is identically zero.  
The external ports are divided into two parts: on the one hand, the pair $(M_{\ext}, \bm{\omega}_{\ext})$, and on the other hand, the pair $(\bm{F}_{1,\rm C}, \bm{v}_{1,\rm C})$.
The latter will be used for interconnection, while the former remains an external port after the interconnection.
Next, we consider the rod of length $\ell_2$ that is connected to the slider at point $B$. The combined mass of the rod and the slider is given by $m_2$ and its center of mass is defined by $r_{2}$. The rod has length $\ell_2$ and its moment of inertia about point B is denoted as $\mathrm{I}_{2,\rm B}$. The orientation of the rod is given by the angle $\bm{\varphi}_2(t) \in \R$, and the position of point~B is described by its coordinates  $\bm{r}_{\mathrm{B}}(t) = \big(\bm{x}_{\mathrm{B}}(t), \bm{y}_{\mathrm{B}}(t) \big)^\top  \in \R^2$. The angular velocity is denoted by $\bm{\omega}_2(t) \in \R$, and the linear velocity in the body-fixed frame is given by $\bm{v}_2(t) = \big(\bm{v}_{2,x'}(t), \bm{v}_{2,y'}(t)\big)^\top \in \R^2$.
An external force $\bm{F}_{\rm ext}(t)\in\R$ is acting horizontally on the slider and forms an external port with the velocity $\bm{v}_{\rm ext}(t)=\bm{v}_{2,x'}(t)$ of the slider. In point $C$, the interconnecting force $\bm{F}_{2,\rm C}(t)\in\R^2$ is acting on the rod. The corresponding velocity is denoted as $\bm{v}_{2,\rm C}(t)\in\R^2$. The dynamics are then given by
\begin{align*}
\begin{pmatrix}
   \dot{\bm{r}}_{\rm B}(t)\\\dot{\bm{\varphi}}_2(t) 
\end{pmatrix}&=\begin{bmatrix}
    \cos(\bm{\varphi}_2(t)) & -\sin(\bm{\varphi}_2(t)) & 0 \\
    \sin(\bm{\varphi}_2(t)) & \cos(\bm{\varphi}_2(t)) & 0 \\
    0 & 0 & 1
\end{bmatrix}\begin{pmatrix}
   \bm{v}_2(t)\\\bm{\omega}_2(t) 
\end{pmatrix},\\
\begin{bmatrix}
        m_2 & 0 & 0 \\
        0 & m_2 & m_2 r_{2} \\
        0 & m_2 r_{2} & \mathrm{I}_{2, \rm B}
    \end{bmatrix}\begin{pmatrix}
    \dot{\bm{v}}_2(t) \\ \dot{\bm{\omega}}_2(t)
\end{pmatrix}
&=\begin{bmatrix}
\phantom{-}\cos(\bm{\varphi}_2(t)) & \sin(\bm{\varphi}_2(t)) \\ -\sin(\bm{\varphi}_2(t)) & \cos(\bm{\varphi}_2(t)) \\ -\ell_2 \sin(\bm{\varphi}_2(t)) & \ell_2 \cos(\bm{\varphi}_2(t)) 
\end{bmatrix}
\bm{F}_{2,\rm C}(t)
\\&\hspace{-2cm} + {\small\begin{bmatrix}
    0 & 0 & m_2 (\bm{v}_{2,y'}(t)+r_2\bm{\omega}_2(t)) \\
    0 & 0 & -m_2 \bm{v}_{2,x'}(t) \\
    -m_2 (\bm{v}_{2,y'}(t)+r_2\bm{\omega}_2(t)) & m_2 \bm{v}_{2,x'}(t) & 0
\end{bmatrix}}\begin{pmatrix}
        \bm{v}_2(t) \\ \bm{\omega}_2(t)
    \end{pmatrix}
\\&\quad +\begin{bmatrix}
    \sin(\bm{\varphi}_2(t)) \\ \cos(\bm{\varphi}_2(t)) \\ 0
\end{bmatrix} \bm{\lambda}(t)
+\begin{bmatrix}
\phantom{-}\cos(\bm{\varphi}_2(t)) \\ -\sin(\bm{\varphi}_2(t)) \\ 0
\end{bmatrix}
\bm{F}_{\ext}(t),\\
0&= \,\begin{bmatrix}0&1&0    
\end{bmatrix}\begin{pmatrix}
   {\bm{r}}_{\rm B}(t)\\{\bm{\varphi}}_2(t) 
\end{pmatrix},\\
\bm{v}_{2,\rm C}(t)&= \begin{bmatrix}
\cos(\bm{\varphi}_2(t)) & -\sin(\bm{\varphi}_2(t)) & -\ell_2 \sin(\bm{\varphi}_2(t)) \\ \sin(\bm{\varphi}_2(t)) & \phantom{-}\cos(\bm{\varphi}_2(t)) & \phantom{-}\ell_2 \cos(\bm{\varphi}_2(t))
\end{bmatrix}\begin{pmatrix}
   \bm{v}_2(t)\\\bm{\omega}_2(t) 
\end{pmatrix},\\
\bm{v}_{\ext}(t)&=\begin{bmatrix}
\cos(\bm{\varphi}_2(t)) & -\sin(\bm{\varphi}_2(t)) & 0
\end{bmatrix}\begin{pmatrix}
   \bm{v}_2(t)\\\bm{\omega}_2(t) 
\end{pmatrix}.
\end{align*}
It can be verified that this system is again of the 
form~\eqref{eq:mks2}, in particular with external port consisting of $\bm{F}_{\ext}(t),\bm{v}_{\ext}(t)$ and $\bm{F}_{2,\rm C}(t),\bm{v}_{2,\rm C}(t)$. 

To model the overall system, that is, the interconnection of rod and crank, the latter quantities can be coupled with the force-velocity pair of the crank via
\[\bm{F}_{1,\rm C}(t)=-\bm{F}_{2,\rm C}(t),\;\;\bm{v}_{1,\rm C}(t)=\bm{v}_{2,\rm C}(t).\]
Condition \eqref{eq:cond-interconn} is satisfied, since the constraint matrix of the combined system
\[
\bm{A}(\bq) = \begin{bmatrix}
    -\ell_1 \sin(\bm{\varphi}_1) & -\cos(\bm{\varphi}_2) &\phantom{-} \sin(\bm{\varphi}_2) & \phantom{-}\ell_2 \sin(\bm{\varphi}_2) \\
    \phantom{-}\ell_1 \cos(\bm{\varphi}_1) & -\sin(\bm{\varphi}_2) & -\cos(\bm{\varphi}_2) & -\ell_2 \cos(\bm{\varphi}_2)
\end{bmatrix}
\]
has always full row rank. 
Based on our findings in Section~\ref{sec:interconn}, the interconnection gives rise to a port-Hamiltonian system with external flows $\bm{\omega}_{\ext}(t)$ and $\bm{v}_{\ext}(t)$, and corresponding external efforts $\bm{M}_{\ext}(t)$ and $\bm{F}_{\ext}(t)$. We omit providing the equations of the overall system.

\begin{remark}
In the presented theory of interconnection for port-Hamiltonian systems, some velocities of the subsystems are used for interconnection.  
However, in the case of the slider-crank mechanism, an interconnection based on position level would be more appropriate.  
We note that coupling via velocities, together with a suitable choice of initial conditions, implicitly yields the desired interconnection on position level.

At the modeling level, interconnection based on  positions requires a modification of the standard port concept and calls for a theoretical framework involving the interconnection of Lagrangian submanifolds rather than Dirac structures.  
An appropriate adaptation of the port concept has already been proposed through the notion of `energy ports'~\cite{KrMavDS24}.
\end{remark}

\section{Conclusion}\label{Sec:Concl}

In this article, we have presented a port-Hamiltonian formulation of a class of multibody systems with rigid components.  
This includes, in particular, the formulation of suitable Dirac structures and Lagrangian submanifolds that allow for the incorporation of kinematic relations, position and velocity constraints, gyroscopic forces, and restoring forces arising from potential energy.  
Furthermore, we have shown that the port-Hamiltonian interconnection of such systems again yields a system of the same structural form.  
In this context, the interconnection relations behave like velocity constraints.  
Our theoretical developments are illustrated by three examples.

\begin{appendices}

\section{Auxiliary results on Dirac structures and Lagrangian submanifolds}\label{Sec:AppA}

We present some results for special modulated Dirac structures and Lagrangian submanifolds that play a key role in this work. 
The necessary theoretical results for the proofs in this appendix are drawn from existing literature. The findings presented here are applied in Sections~\ref{sec:rigidmks} and~\ref{sec:rigidmks2}, which focus on mechanical systems. 

 We begin with a preparatory lemma on local image representations of continuous matrix-valued functions with pointwise full row rank.
 \begin{lemma}\label{lem:contker}
 Let $U\subset\R^k$ be an open set, and let $E:U\to\R^{\ell\times n}$ be a continuous function. Assume that $E$ has constant rank $r\in\N_0$, i.e., $\rank E(x)=r$ for all $x\in U$. Then, for every $x\in U$, there exists a neighborhood $U_{x}\subset U$ of~$x$ and a~continuous function $J:U_{x}\to \R^{n\times(n-r)}$, such that $\im J(y)=\ker E(y)$ for all $y\in U_{x}$.
 \end{lemma}
 \begin{proof}
Fix $x\in U$. Without loss of generality, we assume that $E(x)$ admits a~partition 
\[E(y)=\begin{bmatrix}E_{11}(y)&E_{12}(y)\\E_{21}(y)&E_{22}(y)\end{bmatrix},\quad y\in U,\]
such that $E_{11}(x)\in\R^{r\times r}$ is invertible  (otherwise, rearrange the rows and columns of $E(x)$ accordingly). Since $E$ is continuous and $E_{11}(x)$ is invertibile, there exists a neighborhood $U_{x}\subset U$ of $x$ such that $E_{11}(y)$ is invertible for all $y\in U_{x}$. As a consequence, since $\rank E(y) = r$ for all $y\in U_x$, there exists $R:U_x\to\R^{(\ell-r)\times r}$ such that $[E_{21}(y), E_{22}(y)] = R(y)[E_{11}(y), E_{12}(y)]$ for all $y\in U_x$. Then the function $J:U_{x}\to \R^{n\times(n-r)}$ with
 \[J(y)=\begin{bmatrix}-E_{11}(y)^{-1}E_{12}(y)\\I_{n-r}\end{bmatrix}\]
 has the desired properties.
\end{proof}
 
Next we investigate a special modulated Dirac structure that is derived from another one by incorporating an additional type of constraint. This is used in our analysis of mechanical systems to account for velocity constraints.

\begin{proposition}\label{prop:Dirnonhol}
Let $U\subset \R^k$ be open and $(\mathcal D_x)_{x\in U}$ be a~modulated Dirac structure with $\mathcal D_x\subset\R^n\times \R^n$ for all $x\in U$. Further, let $E:U\to \R^{\ell\times n}$ be continuous and assume that the dimension of the space
\begin{equation}\mathcal{D}_x\cap\big(\R^n\times\ker E(x)\big)\label{eq:spacedimconst}\end{equation}
is independent of $x\in U$.
Then for
\[\widetilde{\mathcal D}_x:=\setdef{\big({f},{e}\big)\in\R^n\times \R^n}{E(x)e=0,\;\exists\, \mu\in\R^\ell:\ \big(f+E(x)^\top\mu,e\big)\in\mathcal{D}_x },\quad x\in U,\]
we have that $(\widetilde{\mathcal D}_x)_{x\in U}$ is a~modulated Dirac structure.
\end{proposition}
\begin{proof}
Fix $x\in U$.\\
{\em Step~1:} We show that $\widetilde{\mathcal D}_x$ is a~Dirac structure. A~straighforward calculation yields that 
\[\forall\,(f,e),(\hat{f},\hat{e})\in \widetilde{\mathcal D}_x:\quad\hat{f}^\top e+f^\top\hat{e}=0.\]
Next assume that $(f,e)\in\R^n\times\R^n$ fulfills $\hat{f}^\top e+f^\top\hat{e}=0$ for all $(\hat{f},\hat{e})\in \widetilde{\mathcal D}_x$. 
We use the orthogonal decomposition $\R^n=\im E(x)^\top + \ker E(x)$ (that is, $(\ker E(x))^\perp = \im E(x)^\top$) to obtain a~decomposition $f=f_1+f_2$, $e=e_1+e_2$ with $f_1,e_1\in \im E(x)^\top$, $f_2,e_2\in \ker E(x)$.
Then, for all $(\hat{f}_1,\hat{e})\in\Dc_x$ with $E(x)\hat{e}=0$, $\hat{\mu}\in\R^\ell$, we have that $(\hat{f}_1+E(x)^\top\hat{\mu}, \hat{e})\in \widetilde{\mathcal D}_x$ and hence
\begin{multline*}
0=  (\hat{f}_1+E(x)^\top\hat{\mu})^\top e+f^\top\hat{e}
=\big(\hat{f}_1+E(x)^\top\hat{\mu})^\top (e_1+e_2)+(f_1+f_2)^\top\underbrace{\hat{e}}_{\in\ker E(x)}\quad\\
=\hat{f}_1^\top (e_1+e_2)
+\hat{\mu}^\top E(x) e_1+f_2^\top{\hat{e}}.\end{multline*}
This gives, by setting $(\hat{f}_1,\hat{e})=(0,0)$ and $\hat\mu=0$ one after the other, that $e_1\in\ker E(x)$ and
\[\forall\,(\hat{f}_1,\hat{e})\in\Dc_x:\quad 0=\hat{f}_1^\top(e_1+e_2)
+f_2^\top{\hat{e}}.
\]
Using $e_1\in\ker E(x)$ together with $e_1\in\im E(x)^\top$, we obtain that $e_1=0$, and thus
\[\forall\,(\hat{f}_1,\hat{e})\in\Dc_x:\quad 0=\hat{f}_1^\top e_2+f_2^\top{\hat{e}}.
\]
Since $\Dc_x$ is a~Dirac structure, this implies that $(f_2,e_2)\in\Dc_x$. Now $e_1=0$ gives $e=e_2$, whence $(f_2,e)\in\Dc_x$.
By $f_1\in \im E(x)^\top$ there exists $\mu\in\R^\ell$ such that $f_1=E(x)^\top\mu$, and thus
\[(f,e)=(f_2+f_1,e)=(f_2+E(x)^\top\mu,e)\in \widetilde{\Dc}_x.\]
{\em Step~2:} We show that there exists
a~neighborhood $U_x\subset U$ of $x$ and a~family $(T_y)_{y\in U_x}$ of linear and bijective mappings $T_y:\R^n\to \mathcal D_y$, such that, for all $z\in\R^n$, the mapping
$y\mapsto T_yz$
    is continuous from $U_x$ to $\R^n\times\R^n$.\\
The definition of a~modulated Dirac structure yields that there exists a neighborhood $U_{1,x}\subset U$ of $x$, and continuous $K,L:U_{1,x}\to \R^{n\times n}$, such that 
\[\forall\, y\in U_{1,x}:\quad \mathcal{D}_y=\im\left[\begin{smallmatrix}K(y)\\L(y)\end{smallmatrix}\right].\] 
Then
\begin{equation}
    \forall\, y\in U_{1,x}:\quad \mathcal{D}_y \cap \big(\R^n\times \ker E(y)\big) =     \left[\begin{smallmatrix}K(y)\\ L(y)\end{smallmatrix}\right] \ker E(y)L(y).
\label{eq:rkconst}
\end{equation}
Since $\dim \mathcal{D}_y=n$, the matrix $\left[\begin{smallmatrix}K(y)\\ L(y)\end{smallmatrix}\right]$ has full column rank. 
Hence, a~combination of \eqref{eq:rkconst} with the rank-nullity theorem yields that
\begin{align*}
\rank E(y)L(y)&=n-\dim\ker E(y)L(y)\\
&=n-\dim \left(\left[\begin{smallmatrix}K(y)\\ L(y)\end{smallmatrix}\right]\ker E(y)L(y)\right)=
n-\dim\Big(\mathcal{D}_y \cap \big(\R^n\times \ker E(y)\big)\Big),
\end{align*}
which is independent of $y$ by the assumption that the last term is constant. Consequently,  $E(y)L(y)$ has constant rank, which we denote by $r_1$. Then, by Lemma~\ref{lem:contker}, there exists a neighborhood $U_{2,x}\subset U_{1,x}$ of~$x$ and some continuous $J_1:U_{2,x}\to\R^{n\times (n-r_1)}$, such that $\im J_1(y)=\ker E(y)L(y)$ for all $y\in U_{2,x}$. The definition of $\widetilde{\Dc}_y$ yields that
\[\forall\,y\in U_{2,x}:\quad \widetilde{\Dc}_y=\im \underbrace{\left[\begin{smallmatrix}K(y)J_1(y)&E(y)^\top\\L(y)J_1(y)&0\end{smallmatrix}\right]}_{=:D(y)\in\R^{2n\times (r_1+\ell)}},\]
which can be seen as follows: If $(f,e)\in \widetilde{\Dc}_y$, then $e\in \ker E(y)$ and $(f+E(y)^\top \mu, e) \in \Dc_y$ for some $\mu\in\R^\ell$, that is
\[
    \begin{pmatrix} f+E(y)^\top \mu \\ e\end{pmatrix} \in \left[\begin{smallmatrix}K(y)\\ L(y)\end{smallmatrix}\right] \ker E(y)L(y)\quad\implies\quad \begin{pmatrix} f \\ e\end{pmatrix}\in \im \left[\begin{smallmatrix}K(y)J_1(y)&E(y)^\top\\L(y)J_1(y)&0\end{smallmatrix}\right].
\]
On the other hand, let $\left(\begin{smallmatrix} f \\ e\end{smallmatrix}\right)\in \im \left[\begin{smallmatrix}K(y)J_1(y)&E(y)^\top\\L(y)J_1(y)&0\end{smallmatrix}\right]$, then
\[
    \begin{pmatrix} f \\ e\end{pmatrix} = \begin{pmatrix} \tilde f \\ \tilde e\end{pmatrix} + \begin{bmatrix} E(y)^\top \\ 0 \end{bmatrix}\mu,\quad \begin{pmatrix} \tilde f \\ \tilde e\end{pmatrix}\in \mathcal{D}_y \cap \big(\R^n\times \ker E(y)\big),\ \mu\in\R^\ell,
\]
from which it follows that $e\in \ker E(y)$ and $(f-E(y)^\top \mu, e) \in \Dc_y$, thus $(f,e)\in \widetilde{\Dc}_y$.

Since $\dim \widetilde{\Dc}_y = n$, we have that $\rank D(y)=n$ for all $y\in U_{2,x}$. Consequently, again using Lemma~\ref{lem:contker}, there exists a~neighborhood $U_{3,x}\subset U_{2,x}$ of $x$, and some continuous $J_2:U_{3,x}\to \R^{(r_1+\ell)\times (r_1+\ell-n)}$, such that $\im J_2(y)=\ker D(y)$ for all $y\in U_{3,x}$.  Consequently, $J_2$ has pointwise full column rank, and another use of Lemma~\ref{lem:contker} yields that there exists a neighborhood  $U_{4,x}\subset U_{3,x}$ of $x$, and some continuous $J_3:U_{4,x}\to \R^{(r_1+\ell)\times n}$, such that $\im J_3(y)=\ker J_2(y)^\top$ for all $y\in U_{3,x}$. Using that $\im J_2(y)$ is the orthogonal complement of $\im J_3(y) = \ker J_2(y)^\top$, we have that $[J_2(y),J_3(y)]$ is an invertible matrix for all $y\in U_{4,x}$. Hence, for all $y\in U_{4,x}$,
\[\widetilde{\Dc}_y=\im D(y)=\im D(y)[J_2(y),J_3(y)]=\im D(y)J_3(y).\]
Since $D(y)J_3(y)\in\R^{2n\times n}$ has full column rank, the desired result holds for $U_x=U_{4,x}$ and $T_y:\R^n\to \mathcal D_y$ with $z\mapsto D(y)J_3(y)z$.
\end{proof}

\begin{remark}\label{rem:Dirnonhol}
In Proposition~\ref{prop:Dirnonhol}, the assumption that the rank of the space $\mathcal{D}_x \cap \big(\R^n \times \ker E(x)\big)$ is constant is essential for $(\widetilde{D}_x)_{x\in U}$ to be a modulated Dirac structure.  
As a counterexample, consider the (constant) Dirac structure  
\[
\mathcal{D} = \im \left[\begin{smallmatrix} 1 & 0 \\ 0 & 0 \\ 0 & 0 \\ 0 & 1 \end{smallmatrix}\right],
\]  
and the matrix-valued function $E:\R \to \R^{1\times 2},\ x\mapsto [1,\,x]$. Then the subspaces $(\widetilde{\Dc}_x)_{x\in\R}$, as constructed in Proposition~\ref{prop:Dirnonhol}, satisfy $\widetilde{\Dc}_0 = \Dc$ and  
\[
\forall\,x\in\R\setminus\{0\}:\quad \widetilde{\Dc}_x = \R^2 \times \{0\}.
\]  
In particular, for any neighborhood $U_0 \subset \R$ of zero and any family $(T_y)_{y\in U_0}$ of linear bijective mappings $T_y: \R^2 \to \widetilde{\Dc}_y$, we find that for $z\in\R^2$ with $T_0 z = (0,0,0,1)$, the mapping $y\mapsto T_y z$ is discontinuous at zero. Therefore, $(\widetilde{\Dc}_x)_{x\in\R}$ is not a modulated Dirac structure.

Indeed, the assumptions of Proposition~\ref{prop:Dirnonhol} do not hold in this case, since $\Dc \cap (\R^2 \times \ker E(0))$ is one-dimensional, whereas $\Dc \cap (\R^2 \times \ker E(x)) = \{0\}$ for all $x \in \R \setminus \{0\}$.

\end{remark}

It was mentioned in Section~\ref{sec:pHsys} that the graphs of gradient fields are Lagrangian submanifolds.
In the following, we present a generalization of this result, where a gradient field is considered alongside a certain type of restriction. This generalization is used in this article to appropriately formulate position constraints.
The following result can essentially be deduced from the findings in \cite[Sec.~4.2]{MvdS20} on so-called {\em Morse families}; see also \cite{BLLCGR19,CastrillonLopez2016}.  
However, in these references, the result is embedded within rather abstract differential geometric concepts.  
For this reason, we have chosen to present an elementary proof here.


\begin{proposition}\label{prop-gradient}
Let $U\subset\R^n$ be open, and let $\mathcal{H}:U\to\R$ and $d:U\to \R^k$, $k,n\in\N_0$, be twice continuously differentiable. Further assume that $d'(x)$ has full row rank for all $x\in U$, and the set $\setdef{x\in U}{d(x)=0}$ is non-empty. Then
\[\mathcal L\coloneqq\setdef{(x,\nabla \mathcal{H}(x)+d'(x)^\top\lambda)}{x\in U\text{ with }d(x)=0,\;\lambda\in\R^k}\subset \R^n\times\R^n\]
 is a~Lagrangian submanifold of $\R^n\times\R^n$.
\end{proposition}
\begin{proof}
It can be seen that $\mathcal L$ is a~differentiable submanifold of $\R^n\times \R^n$. In the sequel we show that it is Lagrangian.\\ 
 \emph{Step 1}: We calculate $T_z \Lc$ for some fixed $z\in \Lc$. First observe that by the assumption that $d'$ has full row rank everywhere it is a submersion, cf.~\cite[Ch.~4]{Lee12}. Then the submersion theorem (see e.g.~\cite[Cor.~5.13]{Lee12}) implies that the non-empty set $M:= \setdef{x\in U}{d(x)=0}$ is a~submanifold with dimension $\dim M = n-k$. Partitioning $z=(z_1,z_2)$, then $z_1\in M$ and there exists a neighborhood $V\subseteq\R^{n-k}$ of zero and a chart $g_1:V\to U$ of~$M$ at~$z_1$, i.e., $g_1(0) = z_1$ and $\ker g_1'(0) = \{0\}$ as well as $\im g_1'(0) = T_{z_1} M$. Furthermore,
 since $d(g_1(v))=0$ for all $v\in V$, by differentiation we obtain that
 \[
     d'(g_1(v)) g_1'(v) = 0.
 \]
 Since $\dim \ker d'(x) = n-k$ for all $x\in U$ by assumption, $g_1(0) = z_1$ and $\ker g_1'(0) = \{0\}$ we find that
 \begin{equation}\label{eq:ker_g=im_psi1}
     \ker d'(z_1) = \im g_1'(0) = T_{z_1} M.
 \end{equation}
 Set $W:= V\times\R^k \subset \R^n$ and define
 \[
     g:W\to \R^n\times\R^n,\ (w_1,w_2)\mapsto \big(g_1(w_1),\nabla \mathcal{H}(g_1(w_1)) - d'(g_1(w_1))^\top (\hat z+w_2)\big),
 \]
 where
 \[
     \hat z := \big(d'(z_1) d'(z_1)^\top\big)^{-1} d'(z_1) \big(\nabla \mathcal{H}(z_1) - z_2\big)\in\R^k.
 \]
 Then $g$ is a chart of $\Lc$ at~$z$, hence
 \[
     T_z \Lc = \im g'(0).
 \]
 We have that
 \[
     g'(w_1,w_2) = \begin{bmatrix} g_1'(w_1)& 0\\
     \nabla^2 \mathcal{H}(g_1(w_1))g_1'(w_1) - D(g_1(w_1)) (\hat z+w_2)g_1'(w_1) & -d'(g_1(w_1))^\top \end{bmatrix} \in\R^{2n\times n},
 \]
 where $\nabla^2 \mathcal{H}$ is the Hessian of $\mathcal{H}$ and $D = ((d')^\top)':\R^n\to L(\R^k,\R^{n\times n})$, where the latter denotes the space of linear mappings from $\R^k$ to $\R^{n\times n}$. 
 Then we find that
 \[
      T_{(z_1,z_2)}\Lc = \setdef{(v_1,v_2)\in\R^n\times\R^n}{\begin{array}{l} v_1 \in \im g_1'(0), \\
      v_2 - \Big(\nabla^2 \mathcal{H}(z_1) - D(z_1) \hat z \Big)  v_1 \in \im d'(z_1)^\top
      \end{array}\!\!}.
 \]
 By the symmetry of second derivatives we have that
 \[
     (D(z_1) \hat z)_{ij} = \sum_{l=1}^k \frac{\partial^2 d_l}{\partial x_i \partial x_j}(z_1) {\hat z}_l = (D(z_1) \hat z)_{ji}
 \]
 for any $i,j=1,\ldots,n$, and hence $D(z_1) \hat z\in\R^{n\times n}$ is symmetric. Furthermore, the Hessian $\nabla^2 \mathcal{H}(z_1)$ is symmetric, hence $M(z_1,z_2) := \nabla^2 \mathcal{H}(z_1) - D(z_1) \hat z$ is symmetric. Finally, by~\eqref{eq:ker_g=im_psi1} we obtain that
 \[
     T_{(z_1,z_2)}\Lc =\setdef{(v_1,v_2)\in\R^n\times\R^n}{v_1\in \ker d'(z_1),\ v_2 - M(z_1,z_2) v_1 \in \left(\ker d'(z_1)\right)^\perp}.
 \]

\noindent\emph{Step 2}: Fix $z=(z_1,z_2)\in\Lc$ and $(v_1,v_2)\in\R^n\times\R^n$. We show that property~\eqref{lag-cond} is satisfied.

 $\Rightarrow$: Assume that $(v_1,v_2)\in T_z \Lc$ and let $(w_1,w_2)\in T_z \Lc$ be arbitrary. Let $x,y\in \left(\ker d'(z_1)\right)^\perp$ be such that
 \[
     v_2 - M(z_1,z_2) v_1 = x,\quad w_2 - M(z_1,z_2) w_1 = y
 \]
 and observe that $v_1^\top y =  w_1^\top x = 0$. Then we have
 \begin{align*}
      v_1^\top w_2 - w_1^\top v_2 &= v_1^\top M(z_1,z_2) w_1 +  v_1^\top y - w_1^\top M(z_1,z_2) v_1 - w_1^\top x \\
     &= v_1^\top M(z_1,z_2) w_1  - w_1^\top M(z_1,z_2) v_1 = 0,
 \end{align*}
 where the last equality follows from symmetry of $M(z_1,z_2)$.

 $\Leftarrow$: Let $w_1\in \ker d'(z_1)$ and $y\in \left(\ker d'(z_1)\right)^\perp$ be arbitrary. Then $(w_1, y +  M(z_1,z_2) w_1) \in T_z \Lc$ and hence we have
 \begin{align*}
    0&= v_1^\top w_2- w_1^\top v_2 \\
    &=  v_1^\top M(z_1,z_2) w_1 +  v_1^\top y - w_1^\top v_2 \\
     &= \big(M(z_1,z_2) v_1 -v_2\big)^\top  w_1  +  v_1^\top y.
 \end{align*}
 Since $w_1$ and $y$ are arbitrary we obtain that $M(z_1,z_2) v_1 -v_2\in\left(\ker d'(z_1)\right)^\perp$ and $v_1\in\ker d'(z_1)$, thus $(v_1,v_2)\in T_z \Lc$. This completes the proof.
 \end{proof}

\end{appendices}

\backmatter

\section*{Statements and Declarations}

\subsection*{Funding}
This work was supported by the Deutsche Forschungsgemeinschaft (DFG, German Research Foundation) through the project \emph{`Adaptive control of coupled rigid and flexible multibody systems with port-Hamiltonian structure'} (Project-ID~362536361).

\subsection*{Author Contributions}
All authors contributed equally to this work.

\subsection*{Competing Interests} The authors report no conflicts of interest.


\bibliography{pH-Multibody}


\begin{thebibliography}{34}
\ifx \bisbn   \undefined \def \bisbn  #1{ISBN #1}\fi
\ifx \binits  \undefined \def \binits#1{#1}\fi
\ifx \bauthor  \undefined \def \bauthor#1{#1}\fi
\ifx \batitle  \undefined \def \batitle#1{#1}\fi
\ifx \bjtitle  \undefined \def \bjtitle#1{#1}\fi
\ifx \bvolume  \undefined \def \bvolume#1{\textbf{#1}}\fi
\ifx \byear  \undefined \def \byear#1{#1}\fi
\ifx \bissue  \undefined \def \bissue#1{#1}\fi
\ifx \bfpage  \undefined \def \bfpage#1{#1}\fi
\ifx \blpage  \undefined \def \blpage #1{#1}\fi
\ifx \burl  \undefined \def \burl#1{\textsf{#1}}\fi
\ifx \doiurl  \undefined \def \doiurl#1{\url{https://doi.org/#1}}\fi
\ifx \betal  \undefined \def \betal{\textit{et al.}}\fi
\ifx \binstitute  \undefined \def \binstitute#1{#1}\fi
\ifx \binstitutionaled  \undefined \def \binstitutionaled#1{#1}\fi
\ifx \bctitle  \undefined \def \bctitle#1{#1}\fi
\ifx \beditor  \undefined \def \beditor#1{#1}\fi
\ifx \bpublisher  \undefined \def \bpublisher#1{#1}\fi
\ifx \bbtitle  \undefined \def \bbtitle#1{#1}\fi
\ifx \bedition  \undefined \def \bedition#1{#1}\fi
\ifx \bseriesno  \undefined \def \bseriesno#1{#1}\fi
\ifx \blocation  \undefined \def \blocation#1{#1}\fi
\ifx \bsertitle  \undefined \def \bsertitle#1{#1}\fi
\ifx \bsnm \undefined \def \bsnm#1{#1}\fi
\ifx \bsuffix \undefined \def \bsuffix#1{#1}\fi
\ifx \bparticle \undefined \def \bparticle#1{#1}\fi
\ifx \barticle \undefined \def \barticle#1{#1}\fi
\bibcommenthead
\ifx \bconfdate \undefined \def \bconfdate #1{#1}\fi
\ifx \botherref \undefined \def \botherref #1{#1}\fi
\ifx \url \undefined \def \url#1{\textsf{#1}}\fi
\ifx \bchapter \undefined \def \bchapter#1{#1}\fi
\ifx \bbook \undefined \def \bbook#1{#1}\fi
\ifx \bcomment \undefined \def \bcomment#1{#1}\fi
\ifx \oauthor \undefined \def \oauthor#1{#1}\fi
\ifx \citeauthoryear \undefined \def \citeauthoryear#1{#1}\fi
\ifx \endbibitem  \undefined \def \endbibitem {}\fi
\ifx \bconflocation  \undefined \def \bconflocation#1{#1}\fi
\ifx \arxivurl  \undefined \def \arxivurl#1{\textsf{#1}}\fi
\csname PreBibitemsHook\endcsname

\bibitem[\protect\citeauthoryear{Jeltsema and {van der Schaft}}{2014}]{JvdS14}
\begin{barticle}
\bauthor{\bsnm{Jeltsema}, \binits{D.}},
\bauthor{\bsnm{{van der Schaft}}, \binits{A.J.}}:
\batitle{Port-{H}amiltonian systems theory: An introductory overview}.
\bjtitle{Foundations and Trends in Systems and Control}
\bvolume{1}(\bissue{2-3}),
\bfpage{173}--\blpage{387}
(\byear{2014})
\doiurl{10.1561/2600000002}
\end{barticle}
\endbibitem

\bibitem[\protect\citeauthoryear{{van der Schaft}}{2017}]{vdS17}
\begin{bbook}
\bauthor{\bsnm{{van der Schaft}}, \binits{A.J.}}:
\bbtitle{$L\sb 2$-Gain and Passivity Techniques in Nonlinear Control},
\bedition{3rd} edn.
\bsertitle{Lecture Notes in Control and Information Sciences}.
\bpublisher{Springer},
\blocation{London}
(\byear{2017}).
\doiurl{10.1007/978-3-319-49992-5} .
\burl{https://doi.org/10.1007/978-3-319-49992-5}
\end{bbook}
\endbibitem

\bibitem[\protect\citeauthoryear{{van der Schaft}}{2013}]{vdS13}
\begin{bchapter}
\bauthor{\bsnm{{van der Schaft}}, \binits{A.J.}}:
\bctitle{Port-{H}amiltonian differential-algebraic systems}.
In: \beditor{\bsnm{Ilchmann}, \binits{A.}},
\beditor{\bsnm{Reis}, \binits{T.}} (eds.)
\bbtitle{Surveys in Differential-Algebraic Equations I}.
\bsertitle{Differential-Algebraic Equations Forum},
pp. \bfpage{173}--\blpage{226}.
\bpublisher{Springer},
\blocation{Berlin Heidelberg}
(\byear{2013}).
\doiurl{10.1007/978-3-642-34928-7_5} .
\burl{https://doi.org/10.1007/978-3-642-34928-7_5}
\end{bchapter}
\endbibitem

\bibitem[\protect\citeauthoryear{{van der Schaft}}{2010}]{vdS10}
\begin{barticle}
\bauthor{\bsnm{{van der Schaft}}, \binits{A.J.}}:
\batitle{Characterization and partial synthesis of the behavior of resistive circuits at their terminals}.
\bjtitle{Systems \& Control Letters}
\bvolume{59}(\bissue{7}),
\bfpage{423}--\blpage{428}
(\byear{2010})
\doiurl{10.1016/j.sysconle.2010.05.005}
\end{barticle}
\endbibitem

\bibitem[\protect\citeauthoryear{Beattie et~al.}{2018}]{BMXZ18}
\begin{botherref}
\oauthor{\bsnm{Beattie}, \binits{C.}},
\oauthor{\bsnm{Mehrmann}, \binits{V.}},
\oauthor{\bsnm{Xu}, \binits{H.}},
\oauthor{\bsnm{Zwart}, \binits{H.}}:
Linear port-{H}amiltonian descriptor systems.
Math.\ Control Signals Syst.
\textbf{30}(4)
(2018)
\doiurl{10.1007/s00498-018-0223-3} .
Article: 17
\end{botherref}
\endbibitem

\bibitem[\protect\citeauthoryear{Mehl et~al.}{2018}]{MMW18}
\begin{barticle}
\bauthor{\bsnm{Mehl}, \binits{C.}},
\bauthor{\bsnm{Mehrmann}, \binits{V.}},
\bauthor{\bsnm{Wojtylak}, \binits{M.}}:
\batitle{Linear algebra properties of dissipative {H}amiltonian descriptor systems}.
\bjtitle{SIAM J.\ Matrix Anal.\ Appl.}
\bvolume{39}(\bissue{3}),
\bfpage{1489}--\blpage{1519}
(\byear{2018})
\doiurl{10.1137/18M1164275}
\end{barticle}
\endbibitem

\bibitem[\protect\citeauthoryear{Maschke and {van der Schaft}}{2018}]{MvdS18}
\begin{barticle}
\bauthor{\bsnm{Maschke}, \binits{B.}},
\bauthor{\bsnm{{van der Schaft}}, \binits{A.J.}}:
\batitle{Generalized port-{H}amiltonian {DAE} systems}.
\bjtitle{Systems \& Control Letters}
\bvolume{121},
\bfpage{31}--\blpage{37}
(\byear{2018})
\doiurl{10.1016/j.sysconle.2018.09.008}
\end{barticle}
\endbibitem

\bibitem[\protect\citeauthoryear{Maschke and {van der Schaft}}{2020}]{MvdS20}
\begin{barticle}
\bauthor{\bsnm{Maschke}, \binits{B.}},
\bauthor{\bsnm{{van der Schaft}}, \binits{A.J.}}:
\batitle{{D}irac and {L}agrange algebraic constraints in nonlinear port-{H}amiltonian systems}.
\bjtitle{Vietnam J. Math}
\bvolume{20},
\bfpage{929}--\blpage{939}
(\byear{2020})
\end{barticle}
\endbibitem

\bibitem[\protect\citeauthoryear{Venkatraman and {van der Schaft}}{2010}]{VvdS10a}
\begin{barticle}
\bauthor{\bsnm{Venkatraman}, \binits{A.}},
\bauthor{\bsnm{{van der Schaft}}, \binits{A.J.}}:
\batitle{Energy shaping of port-{H}amiltonian systems by using alternate passive input-output pairs}.
\bjtitle{Eur. J. Control}
\bvolume{16}(\bissue{6}),
\bfpage{665}--\blpage{677}
(\byear{2010})
\doiurl{10.3166/ejc.16.665-677}
\end{barticle}
\endbibitem

\bibitem[\protect\citeauthoryear{Gernandt et~al.}{2020}]{GeHaRe20}
\begin{barticle}
\bauthor{\bsnm{Gernandt}, \binits{H.}},
\bauthor{\bsnm{Haller}, \binits{F.E.}},
\bauthor{\bsnm{Reis}, \binits{T.}}:
\batitle{A linear relation approach to port-{H}amiltonian differential-algebraic equations}.
\bjtitle{SIAM J.\ Matrix Anal. Appl.}
\bvolume{42}(\bissue{2}),
\bfpage{1011}--\blpage{1044}
(\byear{2020})
\doiurl{10.1137/20M1371166}
\end{barticle}
\endbibitem

\bibitem[\protect\citeauthoryear{Mehrmann and {van der Schaft}}{2023a}]{MS22a}
\begin{barticle}
\bauthor{\bsnm{Mehrmann}, \binits{V.}},
\bauthor{\bsnm{{van der Schaft}}, \binits{A.J.}}:
\batitle{Linear port-{H}amiltonian {DAE} systems revisited}.
\bjtitle{Systems \& Control Letters}
\bvolume{177},
\bfpage{105564}
(\byear{2023})
\end{barticle}
\endbibitem

\bibitem[\protect\citeauthoryear{Mehrmann and {van der Schaft}}{2023b}]{MS23}
\begin{botherref}
\oauthor{\bsnm{Mehrmann}, \binits{V.}},
\oauthor{\bsnm{{van der Schaft}}, \binits{A.J.}}:
Differential-algebraic systems with dissipative {H}amiltonian structure.
Math.\ Control Signals Syst.,
1--44
(2023)
\end{botherref}
\endbibitem

\bibitem[\protect\citeauthoryear{Gernandt et~al.}{2020}]{GeHaReVdS20}
\begin{barticle}
\bauthor{\bsnm{Gernandt}, \binits{H.}},
\bauthor{\bsnm{Haller}, \binits{F.E.}},
\bauthor{\bsnm{Reis}, \binits{T.}},
\bauthor{\bsnm{{van der Schaft}}, \binits{A.J.}}:
\batitle{Port-{H}amiltonian formulation of nonlinear electrical circuits}.
\bjtitle{J. Geom. Phys.}
\bvolume{159},
\bfpage{103959}
(\byear{2020})
\doiurl{10.1016/j.geomphys.2020.103959}
\end{barticle}
\endbibitem

\bibitem[\protect\citeauthoryear{Cervera et~al.}{2007}]{CvdSB07}
\begin{barticle}
\bauthor{\bsnm{Cervera}, \binits{J.}},
\bauthor{\bsnm{{van der Schaft}}, \binits{A.J.}},
\bauthor{\bsnm{Ba{\~{n}}os}, \binits{A.}}:
\batitle{Interconnection of port-{H}amiltonian systems and composition of {D}irac structures}.
\bjtitle{Automatica}
\bvolume{43}(\bissue{2}),
\bfpage{212}--\blpage{225}
(\byear{2007})
\doiurl{10.1016/j.automatica.2006.08.014}
\end{barticle}
\endbibitem

\bibitem[\protect\citeauthoryear{Behrndt et~al.}{2010}]{BKvdSZ10}
\begin{barticle}
\bauthor{\bsnm{Behrndt}, \binits{J.}},
\bauthor{\bsnm{Kurula}, \binits{M.}},
\bauthor{\bsnm{{van der Schaft}}, \binits{A.J.}},
\bauthor{\bsnm{Zwart}, \binits{H.}}:
\batitle{{D}irac structures and their composition on {H}ilbert spaces}.
\bjtitle{J.\ Math.\ Anal.\ Appl.}
\bvolume{372}(\bissue{2}),
\bfpage{402}--\blpage{422}
(\byear{2010})
\doiurl{10.1016/j.jmaa.2010.07.004}
\end{barticle}
\endbibitem

\bibitem[\protect\citeauthoryear{Skrepek et~al.}{2023}]{SkJaEh23}
\begin{botherref}
\oauthor{\bsnm{Skrepek}, \binits{N.}},
\oauthor{\bsnm{J\"aschke}, \binits{J.}},
\oauthor{\bsnm{Ehrhardt}, \binits{M.}}:
Mixed-dimensional geometric coupling of port-{H}amiltonian systems.
Appl. Math. Lett.
\textbf{108508}
(2023)
\doiurl{10.1016/j.aml.2022.108508}
\end{botherref}
\endbibitem

\bibitem[\protect\citeauthoryear{Arnold}{1989}]{Arn89}
\begin{bbook}
\bauthor{\bsnm{Arnold}, \binits{V.I.}}:
\bbtitle{Mathematical Methods of Classical Mechanics},
\bedition{2nd} edn.
\bpublisher{Springer},
\blocation{New York}
(\byear{1989})
\end{bbook}
\endbibitem

\bibitem[\protect\citeauthoryear{Schiehlen and Eberhard}{2014}]{SchiehlenEberhard14}
\begin{bbook}
\bauthor{\bsnm{Schiehlen}, \binits{W.}},
\bauthor{\bsnm{Eberhard}, \binits{P.}}:
\bbtitle{Applied Dynamics},
\bedition{1st} edn.
\bpublisher{Springer},
\blocation{Cham}
(\byear{2014})
\end{bbook}
\endbibitem

\bibitem[\protect\citeauthoryear{Woernle}{2024}]{Woernle24}
\begin{bbook}
\bauthor{\bsnm{Woernle}, \binits{C.}}:
\bbtitle{Multibody {S}ystems~-- {A}n {I}ntroduction to the {K}inematics and {D}ynamics of {S}ystems Of {R}igid {B}odies},
\bedition{1st} edn.
\bpublisher{Springer},
\blocation{Berlin, Heidelberg}
(\byear{2024})
\end{bbook}
\endbibitem

\bibitem[\protect\citeauthoryear{Shabana}{2020}]{Shabana20}
\begin{bbook}
\bauthor{\bsnm{Shabana}, \binits{A.}}:
\bbtitle{Dynamics of Multibody Systems}.
\bpublisher{Cambridge University Press},
\blocation{Cambridge}
(\byear{2020})
\end{bbook}
\endbibitem

\bibitem[\protect\citeauthoryear{Reis}{2021}]{R21}
\begin{barticle}
\bauthor{\bsnm{Reis}, \binits{T.}}:
\batitle{Some notes on port-{H}amiltonian systems on {B}anach spaces}.
\bjtitle{IFAC-Papers\-OnLine}
\bvolume{54}(\bissue{19}),
\bfpage{223}--\blpage{229}
(\byear{2021})
\doiurl{10.1016/j.ifacol.2021.11.082}
\end{barticle}
\endbibitem

\bibitem[\protect\citeauthoryear{Courant}{1990}]{Cou90}
\begin{barticle}
\bauthor{\bsnm{Courant}, \binits{T.J.}}:
\batitle{{D}irac manifolds}.
\bjtitle{Trans. Amer. Math. Soc.}
\bvolume{319},
\bfpage{631}--\blpage{661}
(\byear{1990})
\doiurl{10.1090/S0002-9947-1990-0998124-1}
\end{barticle}
\endbibitem

\bibitem[\protect\citeauthoryear{Lang}{1999}]{Lang1999}
\begin{bbook}
\bauthor{\bsnm{Lang}, \binits{S.}}:
\bbtitle{Fundamentals of Differential Geometry},
\bedition{1st} edn.
\bsertitle{Graduate Texts in Mathematics},
vol. \bseriesno{191}.
\bpublisher{Springer},
\blocation{New York}
(\byear{1999}).
\doiurl{10.1007/978-1-4612-0541-8}
\end{bbook}
\endbibitem

\bibitem[\protect\citeauthoryear{Lee}{2012}]{Lee12}
\begin{bbook}
\bauthor{\bsnm{Lee}, \binits{J.M.}}:
\bbtitle{Introduction to Smooth Manifolds},
\bedition{2nd} edn.
\bsertitle{Graduate Texts in Mathematics},
vol. \bseriesno{218}.
\bpublisher{Springer},
\blocation{New York}
(\byear{2012})
\end{bbook}
\endbibitem

\bibitem[\protect\citeauthoryear{Armstrong-Hélouvry et~al.}{1994}]{AH94}
\begin{barticle}
\bauthor{\bsnm{Armstrong-Hélouvry}, \binits{B.}},
\bauthor{\bsnm{Dupont}, \binits{P.E.}},
\bauthor{\bsnm{Canudas-de~Wit}, \binits{C.}}:
\batitle{A survey of models, analysis tools and compensation methods for the control of machines with friction}.
\bjtitle{Automatica}
\bvolume{30}(\bissue{7}),
\bfpage{1083}--\blpage{1138}
(\byear{1994})
\end{barticle}
\endbibitem

\bibitem[\protect\citeauthoryear{Leine and Nijmeijer}{2004}]{LN04}
\begin{bbook}
\bauthor{\bsnm{Leine}, \binits{R.I.}},
\bauthor{\bsnm{Nijmeijer}, \binits{H.}}:
\bbtitle{Dynamics and Bifurcations of Non-smooth Mechanical Systems}.
\bsertitle{Lecture Notes in Applied and Computational Mechanics},
vol. \bseriesno{18}.
\bpublisher{Springer},
\blocation{Berlin}
(\byear{2004})
\end{bbook}
\endbibitem

\bibitem[\protect\citeauthoryear{Fujimoto et~al.}{2015}]{FuTaMa15}
\begin{barticle}
\bauthor{\bsnm{Fujimoto}, \binits{K.}},
\bauthor{\bsnm{Takeuchi}, \binits{T.}},
\bauthor{\bsnm{Matsumoto}, \binits{Y.}}:
\batitle{On port-{H}amiltonian modeling and control of quaternion systems}.
\bjtitle{IFAC-PapersOnLine}
\bvolume{48}(\bissue{13}),
\bfpage{39}--\blpage{44}
(\byear{2015})
\doiurl{10.1016/j.ifacol.2015.10.211}
\end{barticle}
\endbibitem

\bibitem[\protect\citeauthoryear{Smith}{2002}]{Smith2002}
\begin{barticle}
\bauthor{\bsnm{Smith}, \binits{M.C.}}:
\batitle{Synthesis of mechanical networks: the inerter}.
\bjtitle{IEEE Trans. Automat. Control}
\bvolume{47}(\bissue{10}),
\bfpage{1648}--\blpage{1662}
(\byear{2002})
\doiurl{10.1109/TAC.2002.803532}
\end{barticle}
\endbibitem

\bibitem[\protect\citeauthoryear{Jacobs and Yoshimura}{2011}]{JaYo11}
\begin{bchapter}
\bauthor{\bsnm{Jacobs}, \binits{H.O.}},
\bauthor{\bsnm{Yoshimura}, \binits{H.}}:
\bctitle{Interconnection and composition of {D}irac structures for {L}agrange-{D}irac systems}.
In: \bbtitle{2011 50th IEEE Conference on Decision and Control and European Control Conference},
pp. \bfpage{928}--\blpage{933}
(\byear{2011}).
\doiurl{10.1109/CDC.2011.6160480}
\end{bchapter}
\endbibitem

\bibitem[\protect\citeauthoryear{Jacobs et~al.}{2010}]{Jacobs2010}
\begin{barticle}
\bauthor{\bsnm{Jacobs}, \binits{H.}},
\bauthor{\bsnm{Yoshimura}, \binits{H.}},
\bauthor{\bsnm{Marsden}, \binits{J.E.}}:
\batitle{Interconnection of {L}agrange-{D}irac dynamical systems for electric circuits}.
\bjtitle{AIP Conf. Proc.}
\bvolume{1281},
\bfpage{566}--\blpage{569}
(\byear{2010})
\doiurl{10.1063/1.3498539}
\end{barticle}
\endbibitem

\bibitem[\protect\citeauthoryear{Arnold}{1961}]{Arnold61}
\begin{bbook}
\bauthor{\bsnm{Arnold}, \binits{R.N.}}:
\bbtitle{Gyrodynamics and Its Engineering Applications}.
\bpublisher{Academic Press},
\blocation{New York, NY, London}
(\byear{1961})
\end{bbook}
\endbibitem

\bibitem[\protect\citeauthoryear{Krha{\v{c}} et~al.}{2024}]{KrMavDS24}
\begin{barticle}
\bauthor{\bsnm{Krha{\v{c}}}, \binits{K.}},
\bauthor{\bsnm{Maschke}, \binits{B.}},
\bauthor{\bsnm{{van der Schaft}}, \binits{A.J.}}:
\batitle{Port-{H}amiltonian systems with energy and power ports}.
\bjtitle{IFAC-PapersOnLine}
\bvolume{58}(\bissue{6}),
\bfpage{280}--\blpage{285}
(\byear{2024})
\doiurl{10.1016/j.ifacol.2024.08.294}
\end{barticle}
\endbibitem

\bibitem[\protect\citeauthoryear{Barbero-Li{\~n}\'an et~al.}{2019}]{BLLCGR19}
\begin{barticle}
\bauthor{\bsnm{Barbero-Li{\~n}\'an}, \binits{M.}},
\bauthor{\bsnm{Cendra}, \binits{H.}},
\bauthor{\bsnm{Andr\'es}, \binits{E.}},
\bauthor{\bsnm{Diego}, \binits{D.}}:
\batitle{Morse families and {D}irac systems}.
\bjtitle{J. Geom. Mech.}
\bvolume{11},
\bfpage{487}--\blpage{510}
(\byear{2019})
\end{barticle}
\endbibitem

\bibitem[\protect\citeauthoryear{Castrill{\'o}n~L{\'o}pez and Ratiu}{2016}]{CastrillonLopez2016}
\begin{bchapter}
\bauthor{\bsnm{Castrill{\'o}n~L{\'o}pez}, \binits{M.}},
\bauthor{\bsnm{Ratiu}, \binits{T.S.}}:
\bctitle{{M}orse families and {L}agrangian submanifolds}.
In: \beditor{\bsnm{Castrill{\'o}n~L{\'o}pez}, \binits{M.}},
\beditor{\bsnm{Hern{\'a}ndez~Encinas}, \binits{L.}},
\beditor{\bsnm{Mart{\'i}nez~Gadea}, \binits{P.}},
\beditor{\bsnm{Rosado~Mar{\'i}a}, \binits{M.E.}} (eds.)
\bbtitle{Geometry, Algebra and Applications: From Mechanics to Cryptography}.
\bsertitle{Springer Proceedings in Mathematics \& Statistics},
vol. \bseriesno{161},
pp. \bfpage{65}--\blpage{78}.
\bpublisher{Springer},
\blocation{Cham}
(\byear{2016}).
\doiurl{10.1007/978-3-319-32085-4_6}
\end{bchapter}
\endbibitem

\end{thebibliography}

\end{document}